\documentclass[11pt, psamsfonts,  reqno]{amsart}
\usepackage{amsmath}
\usepackage{amsthm}
\usepackage{amssymb}
\usepackage{amscd}
\usepackage{amsfonts}
\usepackage{amsbsy}
\usepackage{epsfig}
\usepackage{mathrsfs}
\newtheorem{theorem}{Theorem}
\newtheorem{remark}{Remark}
\newtheorem{proposition}{Proposition}
\newtheorem{lemma}{Lemma}
\newtheorem{corollary}{Corollary}
\newtheorem{claim}{Claim}
\newtheorem{definition}{Definition}

\newtheorem{maintheorem}{Theorem}

\def\ie{{\em i.e.,\ }}
\def\eg{{\em e.g.\ }}

\newfont\bbf{msbm10 at 12pt}

\def\eps{\varepsilon}

\def\R{{\mathbb R}}

\def\N{{\mathbb N}}

\def\E{{\mathbb E}}

\def\M{{\mathcal M}}

\def\W{{\mathcal W}}

\def\sm{\setminus}

\def\le{\leqslant}
\def\ge{\geqslant}
\def\htop{h_{\text{top}}}
\def\ptop{P_{\text{top}}}
\def\Pconf{ P_{\mbox{\rm\tiny Conf}} }

\def\dimhyp{\dim_{\text{hyp}}}

\def\drift{\mbox{\bf \it Dr}}

\newdimen\AAdi%
\newbox\AAbo%
\def\AArm{\fam0 }
\def\AAk#1#2{\setbox\AAbo=\hbox{#2}\AAdi=\wd\AAbo\kern#1\AAdi{}}%
\def\AAr#1#2#3{\setbox\AAbo=\hbox{#2}\AAdi=\ht\AAbo\raise#1\AAdi\hbox{#3}}%
\def\1{{\AArm 1\AAk{-.8}{I}I}}%

\begin{document}

\title[Transience and thermodynamics for interval maps]
{Transience and thermodynamic formalism for infinitely
branched interval maps}


\author{Henk Bruin}\address{
Department of Mathematics,
University of Surrey,
Guildford, Surrey, GU2 7XH,
UK }
\email{h.bruin@surrey.ac.uk}
\urladdr{http://personal.maths.surrey.ac.uk/st/H.Bruin/}

\author{Mike Todd}\address{
Mathematical Institute,
University of St Andrews,
North Haugh,
St Andrews,
Fife,
KY16 9SS,
Scotland}
\email{mjt20@st-andrews.ac.uk}
\urladdr{http://www.mcs.st-and.ac.uk/$\sim$miket/index.html}

\date{Version of \today}
\thanks{
The hospitality of the Mittag-Leffler Institute
on Djursholm, Stockholm
(2010 Spring programme on Dynamics and PDEs) is gratefully acknowledged.
MT was partially supported by NSF grants DMS 0606343 and DMS 0908093.  He would also like to thank A Hoffman, N Dobbs and G Iommi for useful conversations in the early stages of this project.
}

\begin{abstract}
We study a one-parameter family of  countably
piecewise linear interval maps, which, although Markov, fail the `large image property'.  This leads to conservative as well as dissipative behaviour for different maps in the family with respect to Lebesgue.  We investigate the transition between these two types, and study the associated thermodynamic formalism, describing in detail the second order phase transitions
(\ie the pressure function is $C^1$ but not $C^2$ at the phase transition) that occur in transition to dissipativity.  We also study the various natural definitions of pressure which arise here, computing these using elementary recurrence relations.
\end{abstract}

\maketitle

\section{Introduction}\label{sec:intro}

The aim of this paper is to understand thermodynamic formalism
of a simple class of infinitely branched uniformly expanding interval maps with suboptimal mixing properties.
Given $\lambda\in (0,1)$, our system is a countably piecewise linear interval map $F_\lambda:(0,1] \to (0,1]$, defined as
\begin{figure}[ht]
\unitlength=4mm
\begin{picture}(9,10)(-10,-0.5) \let\ts\textstyle
\put(-21,6){
$F_\lambda(x):= \begin{cases}\
\frac{x-\lambda}{1-\lambda} & \text{ if } x\in W_1,\\[2mm]
\ \frac{x-\lambda^{n}}{\lambda(1-\lambda)} & \text{ if } x\in W_n, \quad n \ge 2,
\end{cases}$
}
\put(-20.5,2.5){for the intervals $W_n:=(\lambda^{n}, \lambda^{n-1}]$,}
\put(-20.5,1.4){which form a Markov partition.}
\thinlines
\put(0,0){\line(1,0){10}}\put(0,10){\line(1,0){10}}
\put(0,0){\line(0,1){10}} \put(10,0){\line(0,1){10}}
\put(0,0){\line(1,1){10}}
\thicklines
\put(7.5,0){\line(1,4){2.5}} \put(8.3,-0.9){$\tiny W_1$}
\put(5.5,0){\line(1,5){2}} \put(6.0,-0.9){$\tiny W_2$}
\put(4,0){\line(1,5){1.5}}  \put(4.2,-0.9){$\tiny W_3$}
\put(2.9,0){\line(1,5){1.1}}  \put(2.7,-0.9){$\tiny W_4$}
\put(2.1,0){\line(1,5){0.8}}  \put(1.2,-0.9){$\ldots$}
\put(1.52,0){\line(1,5){0.58}}
\put(1.095,0){\line(1,5){0.425}}
\put(0.791,0){\line(1,5){0.304}}
\put(0.572,0){\line(1,5){0.219}}
\put(0.4138,0){\line(1,5){0.1582}}
\put(0.2974,0){\line(1,5){0.1164}}
\put(0.21464,0){\line(1,5){0.08276}}
\end{picture}
\end{figure}

This map was proposed by van Strien to Stratmann as a model for an
induced map of Fibonacci unimodal map. 
Stratmann \& Vogt \cite{StraVogt} computed the Hausdorff dimension of
points that converge to $0$ under iteration of $F_\lambda$
(and in fact this set has full Lebesgue measure for $\lambda > \frac12)$,
which has a bearing on the existence and nature of wild attractors in
interval dynamics, \cite{BKNS}.
Bruin showed (unpublished), that the map $F_\lambda$ is indeed an induced
map of a countably piecewise linear unimodal map, but we intend to come back
to this issue in a forthcoming paper.
The goal of this paper is to investigate the thermodynamic properties of
$((0,1], F_\lambda)$ which is of interest in its own right.
A hint that piecewise expanding maps
with countably many pieces can be Lebesgue dissipative was made early
on by Lasota \& Yorke \cite[page 487]{LY}.
A large part of the current theory of Markov maps with
infinitely many branches relies on a ``large image property'',
which $F_\lambda$ does not satisfy. In contrast,
the distinction between dissipative (transient)
and conservative (recurrent) behaviour leads to
second order phase transition (see below) at 
$t = t_0=\frac{-\log2}{\log\lambda}$
for the `geometric' potential $\Phi_t = -t \log |F'_\lambda|$
(which is assumed to be the
appropriate one-sided derivative at each discontinuity point $\lambda^n$).

Our first main theorem describes the existence
of $(\phi-p)$-conformal reference measures, see Definition~\ref{def:conformal}
for their precise definition. Let
\begin{equation}\label{eq:Pconf}
\Pconf(\phi):=\inf\left\{p\in \R:\text{there exists a } (\phi-p)\text{-conformal measure}\right\}.
\end{equation}
When the potential is $\Phi_t$, for brevity we will also call a 
$(\Phi_t-p)$-conformal measure a $(t,p)$-conformal measure. 

Letting $\psi(t) := \frac{(1-\lambda)^t}{1-\lambda^t}$,
we have the following expression for $\Pconf(\Phi_t)$.

\begin{maintheorem}\label{thm:main diss/cons}
Given $\lambda\in (0,1)$ and $t \in \R$,
$$
\Pconf(\Phi_t) =  \left\{
\begin{array}{ll}
\log \psi(t)  & \text{ if } \lambda^t \le \frac12;\\[2mm]
\log[4\lambda^t(1-\lambda)^t] \qquad & \text{ if } \lambda^t \ge \frac12.
\end{array} \right.
$$
If $p = \Pconf(\phi_t)$ then there exists a $(t,p)$-conformal measure $m_{t,p}$.  This measure is
$$
\left\{ \begin{array}{ll}
\text{conservative } & \text{ if } \lambda^t \le \frac12, \\
\text{dissipative } & \text{ if } \lambda^t > \frac12.
\end{array} \right.
$$
If $p \neq \Pconf(\Phi_t)$, then $m_{t,p}$ is dissipative.
\end{maintheorem}

As we are mostly interested in the case $p = \Pconf(\Phi_t)$, we
 will often abbreviate $m_t = m_{t,p}$ when $p = \Pconf(\Phi_t)$.
We define the \emph{pressure} as
\begin{equation}\label{eq:pressure}
P(\Phi_t) := \sup\left\{ h(\mu) + \int \Phi_t~d\mu : \mu \in \M, \  -\int \Phi_t~d\mu<\infty\right\},
\end{equation}
where the supremum is taken over the set $\M$ of $F$-invariant probability
measures.  A measure $\mu \in \M$ such that $h(\mu) + \int \Phi_t~d\mu=P(\Phi_t)$ is called an \emph{equilibrium state} for $\Phi_t$.

The behaviour of the function $t\mapsto P(\Phi_t)$ is important for understanding the statistical properties of the system.  In the classical hyperbolic case, this function is real analytic \cite{Ru_book}.  We say that the pressure $t\mapsto P(\Phi_t)$ has a \emph{$k$-th order phase transition at $t_0$} if this function is $C^{k-1}$, but not $C^k$ at $t_0$.  In the following theorem, we see that our pressure function has a second order phase transition at $t_0=\frac{-\log2}{\log\lambda}$.

\begin{maintheorem}\label{thm:main pressure formula}
Given $\lambda\in (0,1)$ and $t \in \R$,
$$
P(\Phi_t) =  \left\{
\begin{array}{ll}
\log \psi(t) & \text{ if } \lambda^t \le \frac12;\\[2mm]
\log[4\lambda^t(1-\lambda)^t] \qquad & \text{ if } \lambda^t \ge \frac12,
\end{array} \right.
$$
so there is a second order phase transition at $t_0=\frac{-\log2}{\log\lambda}$.
Moreover, there is an equilibrium state $\mu_t$ for $\Phi_t$ if  $\lambda^t<1/2$.  If such an equilibrium state exists, it is unique and is absolutely continuous w.r.t.\ $m_t$. There is no equilibrium state for $\Phi_t$ when $\lambda^t \ge 1/2$; in particular, there is no measure of maximal entropy.
\end{maintheorem}

For $t = 0$ we arrive at the topological entropy $\htop(F_\lambda) = \log 4$
for all $\lambda \in (0,1)$. It may be a bit surprising that a transitive
map with countably many (expanding) branches has finite entropy,
but this phenomenon has been observed before, \eg \cite{Ruette, MisRai, BoSu}.
The non-existence of a measure of maximal entropy goes back to Gurevich's
paper \cite{Gushiftent}, which basically says that the only
measure of maximal entropy is given by a normalised eigenvector
(with eigenvalue $1$) for the infinite transition matrix associated with the
Markov shift. In our terminology, this is the matrix $A^t$ with $t = 0$
(see \eqref{eq:At}), and the required eigenvector is indeed non-existent
because conformal measure $m_{t,p}$ is dissipative for $t = 0$, $p = \log 4$.
Based on work by Gurevich \cite{Gushiftent}
and Salama \cite{Salama}, Ruette \cite{Ruette} presents
examples of $C^r$ interval maps with infinitely many branches,
finite topological entropy but no measure of maximal entropy.

For the case when the dynamical system $(X, f)$ is a countable Markov shift, and $\phi:X\to \R$ is a sufficiently smooth potential, Sarig  \cite{Sarnull} defined \emph{recurrence}, and its converse, \emph{transience}, in terms of local partition functions (see Section~\ref{sec:pressure}).  If the system is recurrent, then he gave a further condition on such functions under which the  system is \emph{positive recurrent}; the converse of which is \emph{null recurrent}.  He proved that in this context, recurrence is equivalent to the existence of a conservative (see Definition~\ref{def:con} below) $(\phi-P(\phi))$-conformal measure $m$ (see Theorem~\ref{thm:RPF}).  
Moreover, if the system is recurrent, it is positive recurrent if there exists an $f$-invariant probability measure $\mu\ll m$, and null recurrent otherwise.
In \cite{ITtransient}, it was shown that it is reasonable and useful, in order to apply these ideas beyond the realm of shift spaces, to take the conditions on the existence (or non-existence) of such conformal and invariant measures as the definition of the two kinds of recurrence.  Therefore, we can immediately interpret Theorems~\ref{thm:main diss/cons} and \ref{thm:main pressure formula} in terms of recurrence/transience as: $((0,1], F_\lambda, \Phi_t)$ is
\begin{itemize}
\item positive recurrent if $\lambda^t\in (0,1/2)$;
\item null recurrent if $\lambda^t=1/2$;
\item transient if $\lambda^t\in (1/2, 1)$.
\end{itemize}

We can also compute the {\em hyperbolic dimension}
\[
\dimhyp(F_\lambda) := \sup\{ \dim_H(\Lambda) : \Lambda \text{ is
compact, $F_\lambda$-invariant and } \Lambda \not\ni 0 \}
\]
Our abuse of the word hyperbolic here is
motivated by smooth one-dimensional dynamics, where
$0$ is the critical point.
The hyperbolic dimension then refers to taking the supremum over all
invariant closed sets that are bounded away from critical points,
so at every iterate of the map, neighbourhoods of points in hyperbolic
sets map to ``large scale''.
In the usual cases of topologically transitive interval maps this value
is equal to $1$, but the presence of dissipation in our systems can give
 $\dimhyp(F_\lambda)<1$ for $\lambda\in (1/2, 1)$, as in the next theorem.
In addition, for $\lambda\in (0,1)$, we can define the \emph{escaping set} as
$$\Omega_{\lambda}: = \{ x \in [0,1] : \lim_{n \to \infty} F_\lambda^n(x) =
0 \}.$$
The result on the size of the escaping set stated below was proved in \cite{StraVogt}; our more general proof captures the hyperbolic dimension as well.

\begin{maintheorem} The Hausdorff dimension of hyperbolic and escaping sets are
$$
\dimhyp(F_{1-\lambda}) = \dim_H(\Omega_\lambda) =
\left\{ \begin{array}{ll}
- \frac{\log 4}{\log[\lambda(1-\lambda)]} \qquad & \text{ if } \lambda \le \frac12;\\
1 & \text{ if } \lambda \ge \frac12.
\end{array} \right.
$$
\label{thm:main hyp dim}
\end{maintheorem}

Our computations for Theorems~\ref{thm:main diss/cons} and \ref{thm:main pressure formula} use an infinite matrix $A^t$ which models our system as an Markov chain.  There is a corresponding infinite matrix $B^t$ which fits into the transfer operator approach.  For $K\in \N$, we let $A_K^t$ and $B_K^t$ denote the corresponding truncated $K\times K$ matrices and for any matrix $D$ we let $\sigma(D)$ denote the spectral radius of $D$.
In addition we will discuss topological pressure $\ptop$ (based on $(n,\eps)$-separated sets as introduced by Bowen, \cite{Bo_balls}
and then used to define topological pressure in
 \cite{Ru_press} and \cite{Walt_press})
and Gurevich pressure $P_G$, which is particularly
adapted to symbolic countable Markov chains.
The next result brings together these various notions of pressure.  It can be seen as a corollary of Theorems~\ref{thm:main diss/cons} and \ref{thm:main pressure formula}.

\begin{corollary}\label{cor:main all press}
For each $\lambda\in (0,1)$ and $t \in \R$,
\begin{equation}\label{eq:press eq}
\begin{array}{l}
P(\Phi_t) = P_G(\Phi_t)= P_{{\rm top}}(\Phi_t)=\Pconf (\Phi_t) = \log \sigma (B^t) \\[3mm]
\qquad \quad  = \lim_{K\to \infty}\log \sigma (B_K^t) =
\lim_{K\to \infty}\log \sigma (A_K^t).
\end{array}
\end{equation}
If $t = 0$, then the above quantities are all equal to
the topological entropy $\log 4$.
\end{corollary}

One can compare this result to \cite[Proposition 1.2]{PrzR-LSm} for rational maps of the complex plane; specifically the equality between $\Pconf$ and $P$. 

The structure of this paper is as follows.
First, in Section~\ref{sec:acip}, we will prove that $F_\lambda$ has an acip,
\ie an $F_\lambda$-invariant probability measure absolutely continuous w.r.t.\ Lebesgue if and only if $\lambda \in (0, \frac12)$.
We take a probabilistic approach and introduce a random walk on a
Markov chain perspective for these maps.
Continuing the probabilistic approach, in
Section~\ref{sec:Conformal} we introduce the more general
$(t,p)$-conformal measure as a reference measure for $((0,1], F_\lambda)$,
and investigate its thermodynamic properties including what we call
conformal pressure.
For the variational approach to pressure, we need the topological pressure
on $F_\lambda$-invariant compact subsets of $(0,1]$,
and to this end we use infinite matrices matrices $B^t$
and their $K \times K$ cropped versions $B^t_K$ and compute
their leading eigenvalues in Section~\ref{sec:second app}.
This gives us also tools to compute the dimensions
of hyperbolic and escaping sets (Theorem~\ref{thm:main hyp dim})
in Section~\ref{sec:esc hyp dim}.
These various notions of pressure are discussed at length in
 Section~\ref{sec:pressure}, culminating in the proof
of Theorem~\ref{thm:main pressure formula}.
Finally, in Section~\ref{sec:null rec}, we show the null recurrence
of $t$-conformal measure $m_t$ in the case that $\lambda^t = \frac12$.

\section{Acips for $F_\lambda$}
\label{sec:acip}

Now we will calculate the values of $\lambda$ for which there is an $F_\lambda$-invariant probability measure absolutely continuous w.r.t. Lebesgue (acip).

\begin{theorem}
The system $((0,1], F_\lambda)$ has an acip $\mu$ if and only if
$\lambda \in (0,\frac12)$ and in this case
$$
\mu(W_i)= \frac{1-2\lambda}{\lambda} \left( \frac{\lambda}{1-\lambda} \right)^i.$$
If $\lambda \in (\frac12, 1)$, Lebesgue measure is dissipative.
\label{thm:acip}
\end{theorem}

If $\lambda = \frac12$, then Lebesgue measure is conservative, but there is no acip, so the system $((0,1], F_\lambda, \Phi_1)$ is null recurrent. We will return to this case in Section~\ref{sec:null rec}.
The proof of dissipativity is based on a random walk argument, similar
to \cite{BKNS}. Further work in this direction in non-linear setting can 
be found in \cite{MorSma, SchStr}, the latter inspired by
questions in parabolic Kleinian groups.

\begin{proof}
For $F_{\lambda}$ (considered as a Markov process), let $(A_{i,j})_{i,j}$ be the transition matrix corresponding to $F_{\lambda}$,
and let $(v_i)_i$ be the invariant probability vector, \ie
left eigenvector with eigenvalue $1$. As $F_{\lambda}$ is a Markov map, and $F_{\lambda}$ is
linear on each state $W_k$, we obtain $\mu(W_k) = v_k$.
We have
\begin{equation}\label{eq:A}
(A_{i,j})_{i,j} = (1-\lambda)
\left( \begin{array}{ccccccc}
1      & \lambda & \lambda^2 & \lambda^3 & \lambda^4 & \hdots   & \hdots \\
1      & \lambda & \lambda^2 & \lambda^3 & \lambda^4 & \hdots   & \hdots \\
0      & 1      & \lambda   & \lambda^2 & \lambda^3 & \lambda^4 & \hdots \\
0      & 0      & 1        & \lambda   & \lambda^2 & \lambda^3 & \hdots \\
\vdots & \vdots & 0        & 1        & \lambda   & \lambda^2 & \hdots \\
\vdots & \vdots & \vdots   & \vdots   & \vdots   & \vdots   & \ddots
\end{array} \right)
\end{equation}
Suppose $v_i = \beta_i \rho^i$, where $\frac1i \log \beta_i \to 0$ as $i
\to \infty$, so any exponential growth/decline of $v^i$
is captured in $\rho^i$. Then
$$v_N = \sum_{i=1}^{N+1} v_ip_{i,N} =
(1-\lambda) \left( \sum_{i=2}^{N+1} \rho^i\beta_i\lambda^{N+1-i} + \rho
\beta_1 \lambda^{N-1} \right).$$
Dividing by $\lambda^N$ we obtain
\begin{equation}
\beta_N \left(\frac{\rho}{\lambda}\right)^N  = (1-\lambda)
\left( \lambda \sum_{i=2}^{N+1} \beta_i \left(\frac{\rho}{\lambda} \right)^i +
\beta_1 \frac{\rho}{\lambda} \right). \label{12.13}
\end{equation}
Write $\omega = \frac{\rho}{\lambda}$, then subtracting \eqref{12.13}
for $N-1$ from \eqref{12.13} for $N$, and then dividing by $\omega^{N-1}$
gives
$$
(1-\lambda) \lambda \omega^2 \beta_{N+1} - \omega \beta_N + \beta_{N-1} = 0.
$$
Solving the recurrence equation shows that
$\beta_N = \beta_1(\alpha b_+^{N-1} + (1-\alpha) b_-^{N-1})$ for
$\alpha \in \R$ arbitrary, and
$$
b_\pm = \frac{1 \pm \sqrt{1-4\lambda(1-\lambda)} }{(1-\lambda)\lambda \omega}
= \frac{1\pm |1-2\lambda|}{(1-\lambda)\lambda \omega},
\quad \text{ so }
\left\{ \begin{array}{l}
b_+ = \frac{1}{\lambda \omega} = \frac{1}{\rho} \\
b_- = \frac{1}{(1-\lambda) \omega} = \frac{\lambda}{ (1-\lambda)\rho } \\
\end{array} \right.
$$
Therefore $\beta_i$ grows or decreases exponentially
unless $\rho = 1$ or $\rho = \frac{\lambda}{1-\lambda}$.
The former gives $v_i \equiv \beta_1$, which
does not give a probability vector and, moreover,  is only a solution if
$\lambda = \frac12$. The only viable solution is therefore
$\rho = \frac{\lambda}{1-\lambda}$, and direct inspection shows that
taking $\alpha = 1$ and $\beta_i \equiv \beta_1$ indeed solves \eqref{12.13}.
We normalise $\beta_i \equiv (1-\rho)/\rho = (1-2\lambda)/\lambda$ to obtain the normalised solution
\begin{equation}\label{eq:v}
v_i = \frac{1-2\lambda}{\lambda} \left( \frac{\lambda}{1-\lambda} \right)^i
\text{ for } \lambda < \frac12.
\end{equation}
Let us now show that Lebesgue measure is dissipative if $\lambda > \frac12$.
To this end, we consider the action of $F_\lambda$ as a random walk on
the state space $\N$, and let $\chi_n(x) = j$ if $F_{\lambda}^n(x) \in \W_j$.
The probability of going from state $i$ to $j$ is the $i,j$-th entry of $A$,
and we are in particular interested in the conditional expectation
(also called {\em drift})
\begin{equation}\label{eq:drift}
\begin{split} \drift(\lambda) &:= \E(\chi_n-k \mid \chi_{n-1} = k) =
-(1-\lambda) + \lambda(1-\lambda)\sum_{j \ge 1} j\lambda^j \\
&= -(1-\lambda)+ \frac{\lambda^2}{(1-\lambda)} = \frac{2\lambda-1}{1-\lambda}.
\end{split}
\end{equation}
Hence $\E(\chi_n-k \mid \chi_{n-1} = k) > 0$ if $\lambda > \frac12$.
Define $Y_i = (\chi_i-\chi_{i-1}) - \E(\chi_i-\chi_{i-1})$,
then $\E(Y_i) = 0$ and the second moment
$$
\sigma^2 := \E(Y^2_n) =
(1-\lambda)+\lambda(1-\lambda)\sum_{i \ge 1} i^2 \lambda^i
$$
is bounded and independent of $n$.
Thus the Central Limit Theorem applies, so
$\frac{1}{\sigma \sqrt n}\sum_{i=1}^n Y_i$ converges in distribution
to a normally distributed random variable $\mathcal Y$.
Also $\E(\chi_i-\chi_{i-1}) =
\sum_k \E(\chi_i- k | \chi_{i-1} = k) {\mathbb P}(\chi_{i-1} = k)
= \drift(\lambda)$.
Therefore
$$
\chi_n = \chi_0 + \sum_{i=1}^n Y_i + \sum_{i=1}^n \E(\chi_i - \chi_{i-1})
\ge
 \chi_0 + \sigma{\sqrt{n}}\, {\mathcal Y} + n \drift(\lambda)
\to \infty \quad \text{a.s.}
$$
provided $\drift(\lambda) > 0$.
This means that for $\lambda > \frac12$, Lebesgue typical starting points will
have $\chi_n(x) \to \infty$, and $F_\lambda^n(x) \to 0$ as $n \to \infty$.
\end{proof}

\section{Conformal measures and conformal pressure for $F_\lambda$}
\label{sec:Conformal}

In this section we define and compute conformal pressure and combine
it with the drift argument of the previous section to determine whether
or not $(t,p)$-conformal measures are conservative.
Throughout, maps and potentials are assumed to be Borel measurable.

\subsection{Definition of conformal measure}

\begin{definition}\label{def:conformal}
Given a dynamical system $(X,f)$ with potential $\phi:X \to \R$,
a measure $m$ is called \emph{$\phi$-conformal} if
$m(f(A)) = \int_A e^{-\phi} dm$ whenever $f:A \to f(A)$ is one-to-one on a measurable set $A$.
\end{definition}

Notice that if we perform a potential shift by a constant $p$ (\ie
$m(f(A)) = \int_A e^{p-\phi} dm$ whenever $f:A \to f(A)$ is one-to-one),
this will result in a $(\phi-p)$-conformal measure.  Conformal measures corresponding to such shifted potentials are used to define  $\Pconf(\phi)$ defined by \eqref{eq:Pconf}.  Since the canonical class of potentials for our system is $\{-t\log|F'|:t\in \R\}$ (sometimes called the `geometric potentials'), we will be interested in $(-t\log|F'|-p)$-conformal measures for some $p\in \R$.
As mentioned in Section~\ref{sec:intro}, for brevity we will call such a measure a $(t,p)$-conformal
measure and denote it by $m_{t,p}$.

A measure $\mu$ on $X$ is called \emph{non-singular} if $\mu(A)=0$ if and only if $\mu(f^{-1}A)=0.$   A set $ W \subset X$ is called $\emph{wandering}$ if  the sets $\{f^{-n}W\}_{n=0}^{\infty}$ are disjoint.

\begin{definition} \label{def:con} Let $f: X \to X$ be a dynamical system. An $f$-non-singular measure $\mu$ is called \emph{conservative} if every wandering set $W$  is such that
$\mu(W)=0$.
\end{definition}
A conservative measure satisfies the Poincar\'e Recurrence Theorem (see \cite[p.17]{Aaro_book}, or \cite[p.30]{sanotes}).

\subsection{Computation of the conformal pressure for $F_\lambda$}

The log of the function defined in the next lemma will turn out to be the conformal pressure.

\begin{lemma}\label{lem:psi}
Given $\lambda \in [0,1)$, the map
$$
t \mapsto \psi(t) := \frac{(1-\lambda)^t}{1-\lambda^t}
$$
is analytic and strictly decreasing and strictly convex on $(0, \infty)$,
$\lim_{t \to 0} \psi(t) = \infty$,
$\psi(1) = 1$,  $\psi''>0$ and $\lim_{t \to \infty} \psi(t) = 0$.
\end{lemma}

\begin{proof} Straight-forward calculus.
Note that the derivatives are
$$
\psi'(t) = \psi(t) \left[ \log(1-\lambda) + \frac{\lambda^t}{1-\lambda^t} \log \lambda \right] < 0
$$
and
$$
\psi''(t) = \psi(t) \left[ \left(\log(1- \lambda) + \frac{\lambda^t}{1-\lambda^t} \log \lambda\right)^2 +  \frac{\lambda^t}{(1-\lambda^t)^2} \log^2 \lambda \right] > 0
$$
for all $t \in (0, \infty)$.
\end{proof}

\begin{theorem}\label{thm:conformal}
Fix $\lambda\in (0,1)$.  Then for each $t > 0$, the smallest $p\in \R$
such that there is a $(t,p)$-conformal measure $m_{t,p}$
is
$$
p = \Pconf(\Phi_t) = \left\{ \begin{array}{ll}
\log \psi(t) & \text{ if } \lambda^t \le \frac12, \\[1mm]
\log 4[\lambda(1-\lambda]^t & \text{ if } \lambda^t \ge \frac12.
\end{array} \right.
$$
In this case, the conformal measure is given by
$$
m_{t,p}(W_k)= \left\{ \begin{array}{ll}
(1-\lambda^t)\lambda^{t(k-1)} & \text{ if } p = \log \psi(t) \text{ and } \lambda^t \le \frac12, \\[1mm]
\left[ (k-1) + \lambda^{-t}(1-\frac{k}{2}) \right] (\frac12)^k
& \text{ if } p = \log 4[\lambda(1-\lambda)]^t \text{ and } \lambda^t \ge \frac12.
\end{array} \right.
$$
\end{theorem}

\begin{proof}
Suppose that $m_{t, p}$ is a $(t,p)$-conformal measure.
That is $L_{\Phi_t}^*m_{t, p}=e^pm_{t, p}$ for some $p\in \R$.
Then
\begin{align*}
m_{t, p}(W_k)& =\int \1_{W_k}~dm_{t, p}=\int L_{\Phi_t}e^{-p}\1_{W_k}~dm_{t, p} \\
&=\int \sum_{F_\lambda y=x}e^{\Phi_t(y)-p}\1_{W_k}(y)~dm_{t, p}(x)\\
& =\sum_{W_k\to W_j} e^{-p} |F'_\lambda|_{W_k}^{-t}m_{t, p}(W_j),
\label{eq:recur}
\end{align*}
where $W_k\to W_j$ denotes the fact that $F_\lambda$ maps $W_k$ to $W_j$.

Therefore,
\begin{equation}
m_{t, p}(W_1)= e^{-p} (1-\lambda)^{t}\sum_{j\ge 1}m_{t, p}(W_j) \quad \text{ for } k=1
\label{eq:nu_1}
\end{equation}
and
\begin{equation}
m_{t, p}(W_k)= e^{-p}\left[\lambda(1-\lambda)\right]^{t} \sum_{j\ge k-1} m_{t, p}(W_j)
\quad \text{ for } k\ge 2. \label{eq:nu_k}
\end{equation}

As we will see in the following claim, for some values of $(t,p)$, the conformal measure has a particularly simple form.  As we prove below, these are the relevant measures when $\lambda^t<1/2$ and $p=\Pconf(\Phi_t)$.

\begin{claim}\label{Claim1} If $\lambda^t<1$ then
\begin{equation}\label{eq:w}
m_{t, p}(W_k)=(1-\lambda^t)\lambda^{t(k-1)}
\end{equation}
and $p=\log\psi(t)$ solve \eqref{eq:nu_1} and \eqref{eq:nu_k}.
\end{claim}

\begin{proof}
If $m_{t, p}(W_j)=C\gamma^j$, then
 $$C\gamma=m_{t, p}(W_1)= Ce^{-p} (1-\lambda)^{t}\sum_{j\ge 1}\gamma^j=Ce^{-p} (1-\lambda)^{t}\left(\frac\gamma{1-\gamma}\right),$$
and hence $e^{-p}=\frac{1-\gamma}{(1-\lambda)^t}$.
Similarly,
$$
C\gamma^k=m_{t, p}(W_k)= Ce^{-p}\left[\lambda(1-\lambda)\right]^{t} \sum_{j\ge k-1} \gamma^j=Ce^{-p} \left[\lambda(1-\lambda)\right]^{t} \left(\frac{\gamma^{k-1}}{1-\gamma}\right),
$$
and hence $e^{-p}=\frac{\gamma(1-\gamma)}{\left[\lambda(1-\lambda)\right]^{t}}$.  Therefore, $\gamma=\lambda^t$ and $p=\log\psi(t)$.
Finally, taking $C = \frac{1-\lambda^t}{\lambda^t}$ normalises $m_{t, p}$ to
$\sum_k m_{t, p}(W_k) = 1$. This proves Claim~\ref{Claim1}.
\end{proof}

Next we will show that there is no $(t,p)$-conformal measure
if $p < \log 4[\lambda(1-\lambda)]^t$.
Suppose that there is a $(t,p)$-conformal measure $m_{t,p}$.
Let $0 < \eps := \log(4[\lambda(1-\lambda)]^t) - p$.
We will see in the proof of Theorem~\ref{thm:B ev} later on that the number of
periodic points in $W_1$ of period $n$ exceeds $4^{(1-\eps/2)n}$
for $n$ sufficiently large.
Let $X_n = W_1 \cap F^{-n}(W_1)$
and let $\{Y_{n,i}\}_i$ be the collection of connected components of $X_n$.  Since each $Y_{n,i}$ contains a periodic point of period $n$, there exists $C>0$ such that $\#\{Y_{n,i}\}_i\ge 4^{n(1-\eps/2)}$
for $n$ sufficiently large.
To compute the measure of $Y_{n,i}$, we can use
$$
m_{t,p}(W_1) =\int_{Y_{n,i}}e^{pn}\ (1-\lambda)^{-t} \
\left[\lambda(1-\lambda)\right]^{-t(n-1)}~dm_{t,p},
$$
so $m_{t,p}(Y_{n,i}) = C m_{t,p}(W_1)e^{-pn} \left[\lambda(1-\lambda)\right]^{tn}$
for some $C > 0$.
Using the cardinality of components $Y_{n,i}$,
we find
$$
 4^{n(1-\eps/2)} \cdot C  e^{-pn} \left[\lambda(1-\lambda)\right]^{tn} m_{t,p}(W_1)
\le m_{t,p}(W_1),
$$
but the condition $p < \log 4[\lambda(1-\lambda]^t$ makes this impossible
to satisfy for large $n$.

From now on assume that $p \ge \log 4[\lambda(1-\lambda]^t$.
Subtracting \eqref{eq:nu_k} for $k+1$ from
\eqref{eq:nu_k} for $k$ gives for $x_k := m_{t, p}(W_k)$:
$$
x_{k+1}-x_k+cx_{k-1}=0 \quad \text{ where } c:=e^{-p}[\lambda(1-\lambda)]^t.
$$
Thus the roots of the generating equation $r^2-r+c=0$
are $r_\pm=\frac{1\pm\sqrt{1-4c}}2$, and the general solution is
$x_k = A_+ r_+^k + A_- r_-^k$.
We normalise so that the total mass is $\sum_k x_k = 1$.
Then \eqref{eq:nu_1} and \eqref{eq:nu_k} for $k=2$ give
\begin{equation}\label{eq:x1x2}
\left\{ \begin{array}{rcccl}
x_1 &=& e^{-p} (1-\lambda)^t &=& c \lambda^{-t}, \\
x_2 &=& e^{-p} \lambda^t (1-\lambda)^t &=& c.
\end{array} \right.
\end{equation}
Substituting $x_k = A_+ r_+^k + A_- r_-^k$ in this equation and solving
for $A_\pm$ gives
$$
\left\{ \begin{array}{rcc}
c &=&  A_+ \lambda^t r_+ + A_- \lambda^t r_- \\[2mm]
c &=& A_+ r_+^2 + A_- r_-^2
\end{array} \right.
\quad \text{ and } \quad
\left\{ \begin{array}{rcl}
A_+ &=& \frac{c(\lambda^t-r_-)}{\lambda^t(r_+-2c)}, \\[2mm]
A_- &=& \frac{c(r_+-\lambda^t)}{\lambda^t(2c-r_-)}.
\end{array} \right.
$$
The form $r_\pm=\frac{1\pm\sqrt{1-4c}}2$ and $c = e^{-p}[\lambda(1-\lambda)]^t$ implies the inequalities
$$
0 < r_- < 2c = 2e^{-p} \lambda^t(1-\lambda)^t \le \frac12 < r_+ < 1
\quad \text{ and } \quad
\lambda^t < r_+,
$$
where the third inequality is strict if $p > \log 4[\lambda(1-\lambda)]^t$.
This shows that $A_- \ge 0$.

If $\lambda^t \le \frac12$, then we find $A_+ \ge 0$ when
$\lambda^t \ge r_-$,
which is precisely the case when $p \ge \psi(t)$.
It is under this condition that $x_k > 0$ for all $k \in \N$.
If $\lambda^t = r_-$ (so $p = \log \psi(t)$), then $A_+ = 0$, and
$A_- = \frac{1-\lambda^t}{\lambda^t}$.
Hence the solution is the one given in Claim~\ref{Claim1}.

If $\lambda^t > \frac12$, then clearly $\lambda^t > r_-$
and hence $A_+ > 0$ for all allowed values of $p$, and
$p = \log 4[\lambda(1-\lambda)]^t$ is smallest of these.
In this case $r_+ = r_- = \frac12$, and this double root
leads to a solution $x_k = A (\frac12)^k + Bk (\frac12)^k$.
Substitution in \eqref{eq:x1x2} gives
$A = 1-\frac{1}{\lambda^t}$ and $B = \frac{1}{2\lambda^t} - 1$.
\end{proof}

\begin{proof}[Proof of Theorem~\ref{thm:main diss/cons}]
The value of the smallest $p\in \R$ for which there is a $(t,p)$-conformal
measure follows from Theorem~\ref{thm:conformal}.
Now that we have established the existence of a $(t,\Pconf(\Phi_t))$-conformal
measure $m_{t}=m_{t,\Pconf(\Phi_t)}$, for $\lambda^t<1/2$ we can extend our transition probability matrix $A$
from \eqref{eq:A} in
Section~\ref{sec:acip} to a transition probability matrix with respect
to $m_t$.
Indeed, measured in $m_t$, the probability to move
from state $W_i$ to $W_j$ is (using the definition of $(t,\log \psi(t))$-conformal measure)
\begin{align}\label{eq:conformal_transitions}
\frac{m_t(W_i\cap F_\lambda^{-1}(W_j)) }{ m_t(W_i) } & =
\frac{|F'|_{W_i}|^{-t}\psi(t)m_t(W_j)}{|F'|_{W_i}|^{-t}\psi(t)\sum_{k\ge i-1}m_t(W_k)} \\
&= \frac{ (1-\lambda^t) \lambda^{t(j-1)} }
{ \sum_{k \ge i-1} (1-\lambda^t) \lambda^{t(k-1)} } \nonumber\\
&= (1-\lambda^t) \lambda^{t(j-i+1)} \nonumber
\end{align}
provided $j \ge i-1$.
Therefore, if
\begin{equation}\label{eq:At}
A^t = (1-\lambda)^t
\left( \begin{array}{ccccccc}
1^t    & \lambda^t & \lambda^{2t} & \lambda^{3t} & \hdots & \hdots   & \hdots \\
1^t    & \lambda^t & \lambda^{2t} & \lambda^{3t} &        &          &   \\
0      & 1^t       & \lambda^t   & \lambda^{2t} & \lambda^{3t}   &   &   \\
0      & 0         & 1^t         & \lambda^t   & \lambda^{2t} &     &  \\
\vdots &           & 0           & 1^t         & \lambda^t   & \lambda^{2t} & \dots \\
       &           &             &             & \ddots      & \ddots    & \ddots
\end{array} \right).
\end{equation}
is the matrix $A$ in \eqref{eq:A} with all entries raised to the power $t$,
then  $\psi^{-1}(t) A^t$ is a probability matrix and
$\frac{m_t(W_i\cap F_\lambda^{-1}(W_j)) }{ m_t(W_i) } =
\psi^{-1}(t) A^t_{i,j}$.
Now we are able to use the drift argument in the proof of
Theorem~\ref{thm:acip}
to conclude that the measure $m_t$ is conservative if 
$\lambda^t<1/2$.

Now to prove that any $(t,p)$-conformal measure is dissipative whenever
$\lambda^t > 1/2$ or $p > \Pconf(\Phi_t)$ (we leave the null recurrent
case $\lambda^t = \frac12$ and $p = \psi(t)$ to
Section~\ref{sec:null rec}),
we use the same drift argument for
$m_{t, p}(W_j) = A_+ r_+^j + A_- r_-^j$ as in the proof of
Theorem~\ref{thm:conformal}.
(Note that $r_\pm$ and hence $A_\pm$ depends on $p$ via
$c = e^{-p}[\lambda(1-\lambda)]^t$. Note also that $A_+ > 0$
for $p > \Pconf(\Phi_t)$.)
Inserting this solution in \eqref{eq:conformal_transitions},
we find that the transition probability for $W_i \to W_j$ is
\begin{align*}
\frac{m_{t, p}(W_i\cap F_\lambda^{-1}(W_j)) }{ m_{t, p}(W_i) } & =
\frac{|F'|_{W_i}|^{-t} e^p m_{t, p}(W_j)}{|F'|_{W_i}|^{-t} e^p\sum_{k\ge i-1}m_{t, p}(W_k)} \\
&= \frac{ A_+ r_+^j +  A_- r_-^j }
{ \sum_{k \ge i-1}  A_+ r_+^j +  A_- r_-^j } \\
&= C_i \left( r_+^{j-i+1} + \alpha r_-^j r_+^{-i+1} \right)
\end{align*}
where $\alpha = \frac{A_-}{A_+} =
\frac{(r_+-\lambda^t)(r_+-2c)}{(\lambda^t-r_+)(2c-r_-)}$,
and
$C_i = (1-r_+) \left( 1 + \alpha \frac{-r_+}{1-r_-} (\frac{r_-}{r_+})^i \right)^{-1} \to 1-r_+$ as $i \to \infty$.
Since $r_- < \frac12 < r_+$, there is $i_0$ such that the drift
\begin{align*}
 C_i \sum_{j \ge i-1} (j-i-1) \left( r_+^{j-i+1} + \alpha r_-^j r_+^{-i+1} \right)
&= C_i \left( \frac{2r_+-1}{(1-r_+)^2} + \alpha  \left(\frac{r_-}{r_+}\right)^{i-1}
\frac{2r_--1}{(1-r_-)^2} \right)
\end{align*}
is positive for all $i \ge i_0$.
This means that whenever an orbit reaches
a state $i \ge i_0$, the probability of wandering
off to infinity afterwards is positive.
Since $\chi_n(x) \ge i_0$ infinitely often $m_{t, p}$-a.e., it follows that
$m_{t, p}$-typical orbits converge to $0$, proving that $m_{t, p}$ is dissipative.
\end{proof}

Next we discuss the $F_\lambda$-invariant measures that are absolutely
continuous w.r.t.\  the $(t,\Pconf(\Phi_t))$-conformal measure $m_t$

\begin{proposition}\label{prop:mut}
If $\lambda^t\in (0, 1/2)$ and $p = \psi(t)$,
then there is an $F_\lambda$-invariant probability
measure $\mu_t \ll m_t$ such that
\begin{equation}\label{eq:vt}
\mu_t(W_j) = v_i^t := \frac{1-2\lambda^t}{\lambda^t} \left(\frac{\lambda^t}{1-\lambda^t}\right)^i.
\end{equation}
\end{proposition}

\begin{proof}
Direct inspection shows that
the probability matrix $\psi^{-1}(t) A^t$ preserves
the probability vector $(v^t_j)_{j \ge 1}$,
provided $\lambda^t \in (0,\frac12)$.
For $\lambda^t \ge \frac12$, the
vector $\underline v^t$ is not summable, and the drift argument  from
Section~\ref{sec:acip} shows that $m_t$ is in fact dissipative for $\lambda^t > \frac12$.
(The case $\lambda^t = \frac12$ is dealt with in Section~\ref{sec:null rec}.)
\end{proof}

\begin{remark}
Observe that in the case that there is an equilibrium state $\mu_t$ for $\Phi_t$, the density $\frac{d\mu_t}{dm_t}$ is unbounded.  Moreover, $\mu_t$ is not a Gibbs state.  These facts can be seen as follows.

Comparing Theorem~\ref{thm:conformal} with formula \eqref{eq:vt},
we obtain that
$\frac{\mu_t(W_n)}{m_t(W_n)}= \frac{1-2\lambda^t}{(1-\lambda^t)^{n+1}}$ and so, noticing that the density $\frac{d\mu_t}{dm_t}$ is constant on 1-cylinders, we have $\frac{d\mu_t}{dm_t}\big|_{W_n} \to \infty$ as $n\to \infty$.

Now let
$$
[e_0 \cdots e_{n-1}]= \{x \in (0,1] :
F^kx \in W_{e_j} \text{ for } 0 \le j < n \}.
$$
be our notation for an $n$-cylinder set.
To show that $\mu_t$ is not a Gibbs measure, we check that there is no
{\em distortion constant} $C \ge 1$ such that for all cylinder sets $[e_0\cdots e_n]$
\begin{equation*} 
\frac{1}C \le \frac{\mu_t([e_0\cdots e_n])}
{\exp \left(-np + S_n\Phi_t(x)\right)} \le C,
\end{equation*}
where $S_n\Phi_t(x) = \sum_{k = 0}^{n-1} \Phi_t(F^k_\lambda))$
is the $n$-th ergodic sum and $p = \Pconf(\Phi_t)$.

Since $m_t$ is $(t,p)$-conformal,
\begin{eqnarray*}
m_t([e_0\cdots e_n]) &=& e^{-np+S_n\Phi_t} m_t( F^n_\lambda([e_0\cdots e_n]) ) \\
 &=& e^{-np+S_n\Phi_t} \frac{1-2\lambda^t}{1-\lambda^t}
\sum_{k \ge e_{n-1}-1} \left( \frac{1-\lambda^t}{\lambda^t} \right)^k \\
& = &
e^{-np+S_n\Phi_t} \left( \frac{1-\lambda^t}{\lambda^t} \right)^{e_{n-1}-2}.
\end{eqnarray*}
Therefore
$\frac{ \mu_t([e_0\cdots e_n]) }{ e^{-np+S_n\Phi_t} }
= \frac{1-2\lambda^t}{(1-\lambda^t)^{1+e_0}} \cdot
\left( \frac{1-\lambda^t}{\lambda^t} \right)^{e_{n-1}-2}$
which is unbounded in $e_{n-1}$ and $e_0$.
\end{remark}

\section{A second approach to thermodynamic formalism for $F_\lambda$}
\label{sec:second app}

In this section we employ the matrix $B^t$ in place of $A^t$ used previously.  As we show in Section~\ref{sec:pressure}, this new matrix is more closely associated to the formalism of Sarig.

There are two different ways of computing the sizes and their sums of
$n$-cylinder sets, w.r.t.\ potential $-t \log |F'_\lambda|$ (or equivalently,
in terms of $t$-dimensional Hausdorff measure).
One way is shown before Proposition~\ref{prop:hypdim}, using the matrix $A^t$,
and is based on first computing the sizes of $1$-cylinders (namely
for every $e_0 = j$, there are two such cylinders, each of
``$t$-dimensional'' length $|W_j|^t$), and then considering
which part of an $n$-cylinder belongs to a
particular $(n+1)$-subcylinder.
If $\underline 1$ is the row-vector of ones and $\underline w_t$
is the row-vector with entries $|W_j|^t$, then this way gives
\begin{equation}\label{eq:dimcover1}
\sum_{e_0\dots e_{n-1}} |[e_0\dots e_{n-1}]|^t =
\underline w_t \cdot (A^t)^{n-1} \cdot \underline 1^T,
\end{equation}
where ${}^T$ stands for the transpose of a vector or matrix.

The other way is by considering the slopes of $F_\lambda^n$ and the
``$t$-dimensional'' length of the $F_\lambda^n$-image of each cylinder.
The matrix $B^t$ contains the inverse slopes (raised to the power $t$)
$$
B^t = (1-\lambda)^t \left( \begin{array}{ccccccc}
1      & 1 & 1 & 1 & 1 & \hdots   & \hdots \\
\lambda^t      & \lambda^t & \lambda^t & \lambda^t & \lambda^t & \hdots   & \hdots \\
0      & \lambda^t      & \lambda^t   & \lambda^t & \lambda^t & \lambda^t & \hdots \\
0      & 0      & \lambda^t        & \lambda^t   & \lambda^t & \lambda^t & \hdots \\
\vdots & \vdots & 0        & \lambda^t        & \lambda^t   & \lambda^t & \hdots \\
\vdots & \vdots & \vdots   & \vdots   & \vdots   & \vdots   & \ddots
\end{array} \right).
$$
This way gives
\begin{equation}\label{eq:dimcover}
\sum_{e_0 \dots e_{n-1}} |[e_0\dots e_{n-1}]|^t =
\underline 1 \cdot (B^t)^{n-1} \cdot \underline w_t^T.
\end{equation}
For this to make sense, it is important for both matrices to be
related,
and in fact, if
$A^t_K$ and $B^t_K$ are the $K \times K$-left-upper matrices of
$A^t$ and $B^t$ respectively, then we have the following result.

\begin{lemma}\label{lem:charA=charB}
For each $K$, the characteristic polynomials of $A^t_K$ and $B^t_K$
coincide.
\end{lemma}

\begin{proof}
The matrix $A^t_K$ can be transformed into the matrix $B^t_K$ by a sequence of
$2(K-1)$ elementary column and row operations. Namely, we multiply the $i$-th
column by $\lambda^{-(i-1)t}$ and the $i$-th row by $\lambda^{(i-1)t}$.
This has no effect on diagonal elements.
Hence, the same operations transform $(A^t_K - xI)$ into $(B^t_K - xI)$,
and the effects on the determinant cancel out.
\end{proof}

This means that we can dispense with the distinction between $A^t$ and
$B^t$ in characteristic polynomials
$$
\alpha_{t, K}(s) = (1-\lambda)^{-t} \det(A^t_K-sI) = (1-\lambda)^{-t}
\det(B^t_K-sI).
$$
In fact, $\alpha_{t, K}$ satisfies the recursive formula
$\alpha_{t, K}(s) = -s\alpha_{t, K-1}(s)  -s\lambda^t
\alpha_{t, K-2}(s)(s)$
which leads to
$$
\alpha_{t, K}(s) = \sum_{j=0}^{\lfloor (K+1)/2 \rfloor}
\binom{K+1-j}{j} (-1)^{K+1-j} s^{n-j} \lambda^{tj},
$$
that is, the coefficients of $\alpha_{0,K}$ are signed elements of
Pascal's triangle
along the $K$-th north-east-east diagonal.
Letting $s_{t,K}$ be the corresponding leading eigenvalue, we have
\[
\begin{array}{rclcrcl}
\alpha_{t,1}(s) &=&  -s+1 & \quad  & s_{t,1} &=& 1 \\[2mm]
\alpha_{t,2}(s) &=&  s^2-(\lambda^t+1)s & \qquad  & s_{t,2}&=&
\lambda^t+1 \\[2mm]
\alpha_{t,3}(s) &=&  -s^3+(2\lambda^t+1)s^2 - \lambda^t s & \qquad  &
s_{t,3} &=& \frac{2\lambda^t+1+ \sqrt{4\lambda^{2t}+1}}{2} \\[2mm]
\alpha_{t,4}(s) &=&  s^4-(3\lambda^t+1)s^3 +(\lambda^{2t}+2 \lambda^t) s^2 & \qquad  &
s_{t,4} &=& \frac{3\lambda^t+1+ \sqrt{5\lambda^{2t}-2\lambda^t+1}}{2}
\\[2mm]
\vdots \quad & & \qquad \vdots & & \vdots\ & &\qquad \vdots
\end{array}
\]

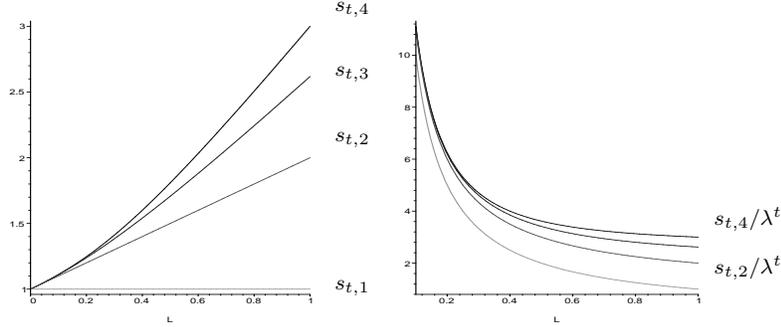
\begin{figure}
\begin{center}
\unitlength=4.3mm
\begin{picture}(25,12)(1,1)
\put(1,1.7){\resizebox{4cm}{4cm}{\includegraphics{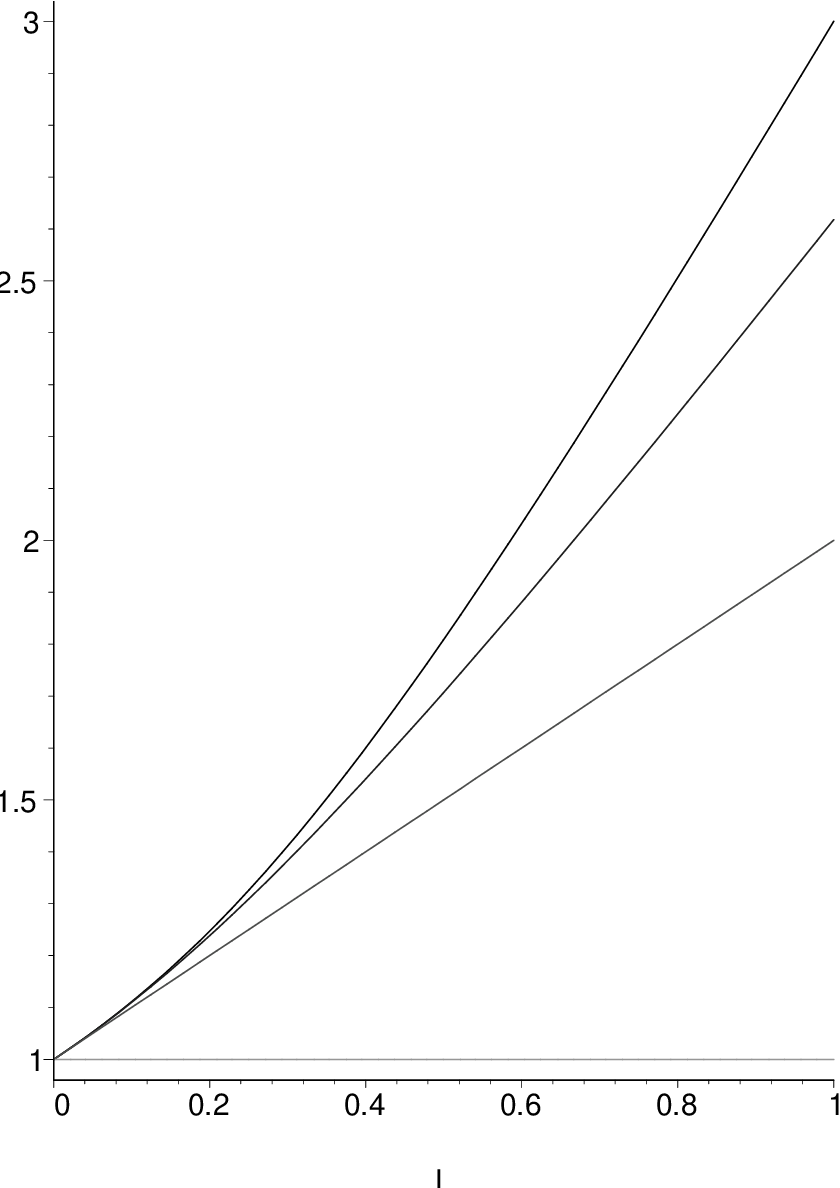}}}
\put(11,2.7){\scriptsize $s_{t,1}$}
\put(11,7.3){\scriptsize $s_{t,2}$}
\put(11,9.3){\scriptsize $s_{t,3}$}
\put(11,11.3){\scriptsize $s_{t,4}$}
\put(13,1.7){\resizebox{4cm}{4cm}{\includegraphics{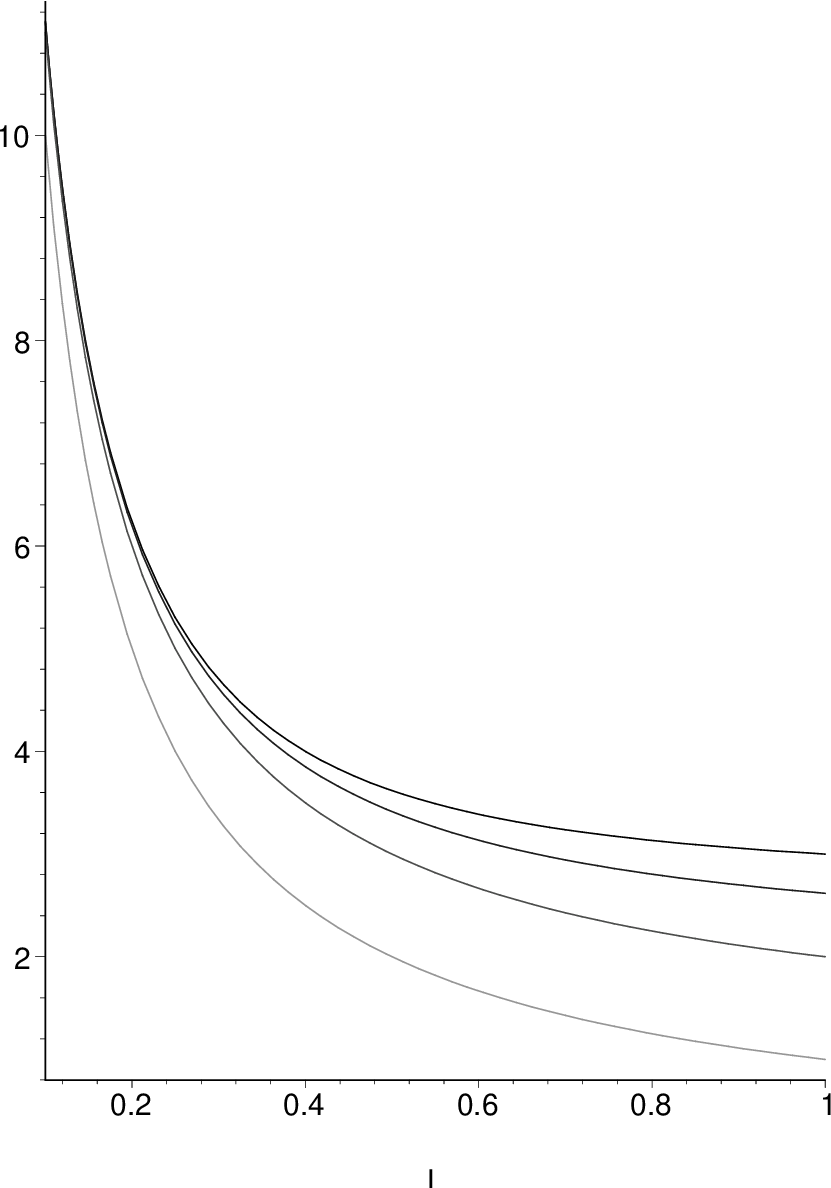}}}
\put(22.73,3.2){\scriptsize $s_{t,2}/\lambda^t$}
\put(22.73,4.7){\scriptsize $s_{t,4}/\lambda^t$}
\end{picture}
\end{center}
\caption{\label{fig:s1to4} Graphs of the eigenvalues $s_{t,j}$ (left) of
$A^t_j$ and $s_{t,j} /\lambda^t$ (right) for $j =
  1, \dots, 4$ as function of $\lambda^t \in [0,1]$.}
\end{figure}

\begin{remark}
For each $K$, the non-negative matrices $A^t_{K+1}$ and $B^t_{K+1}$
are at least as large element-wise
as the matrices $A_K$ and $B_K$ augmented with an extra row and column
of zeroes.
From this it follows that the leading eigenvalues
are increasing in $K$ for all $t$.
\end{remark}

One can verify by induction that $r^{-K} \alpha_{t,K}(r) = (-1)^K\lambda^{Kt}$
for $r = 1/(1-\lambda^t)$, which suggests that the leading root
of $\alpha_{t,K}$ is $x_{t,K}(t) = r(1-\lambda)^t < \psi(t)$,
and potentially in the limit $x(t) = \lim_{K\to\infty} x_{t,K} = \psi(t)$.
The next proposition confirms this for $\lambda^t \le \frac12$.

\begin{theorem}\label{thm:B ev}
The limit of the leading eigenvalues of the matrices $B^t_K$ is
\begin{equation}\label{eq:xt}
x(t) := \lim_{K \to \infty} x_{t,K} = \left\{ \begin{array}{llll}
\psi(t) = \frac{(1-\lambda)^t}{1-\lambda^t} & \text{ if } \lambda^t \le \frac12
& \text{ \ie } t \ge t_0 := \frac{-\log 2}{\log \lambda} ;\\[2mm]
4\lambda^t(1-\lambda)^t & \text{ if } \lambda^t \ge \frac12
& \text{ \ie } t \le t_0.
\end{array} \right.
\end{equation}
Hence $\log x(t)$ is analytic, except at $t = t_0$ where it is $C^1$ but not $C^2$.
Furthermore,
\begin{equation}\label{eq:t1}
\log x(t_1) = 0 \qquad \text{ for } \quad t_1 = \left\{ \begin{array}{ll} 1  & \text{ if } \lambda \le \frac12; \\[1mm]
-\frac{\log 4}{\log[\lambda(1-\lambda)]}  & \text{ if } \lambda \ge \frac12. \end{array} \right.
\end{equation}
\end{theorem}

\begin{figure}
\begin{center}
\unitlength=4.3mm
\begin{picture}(25,12)(3,1)
\put(1,1.7)
{\resizebox{5.5cm}{4.5cm}{\includegraphics{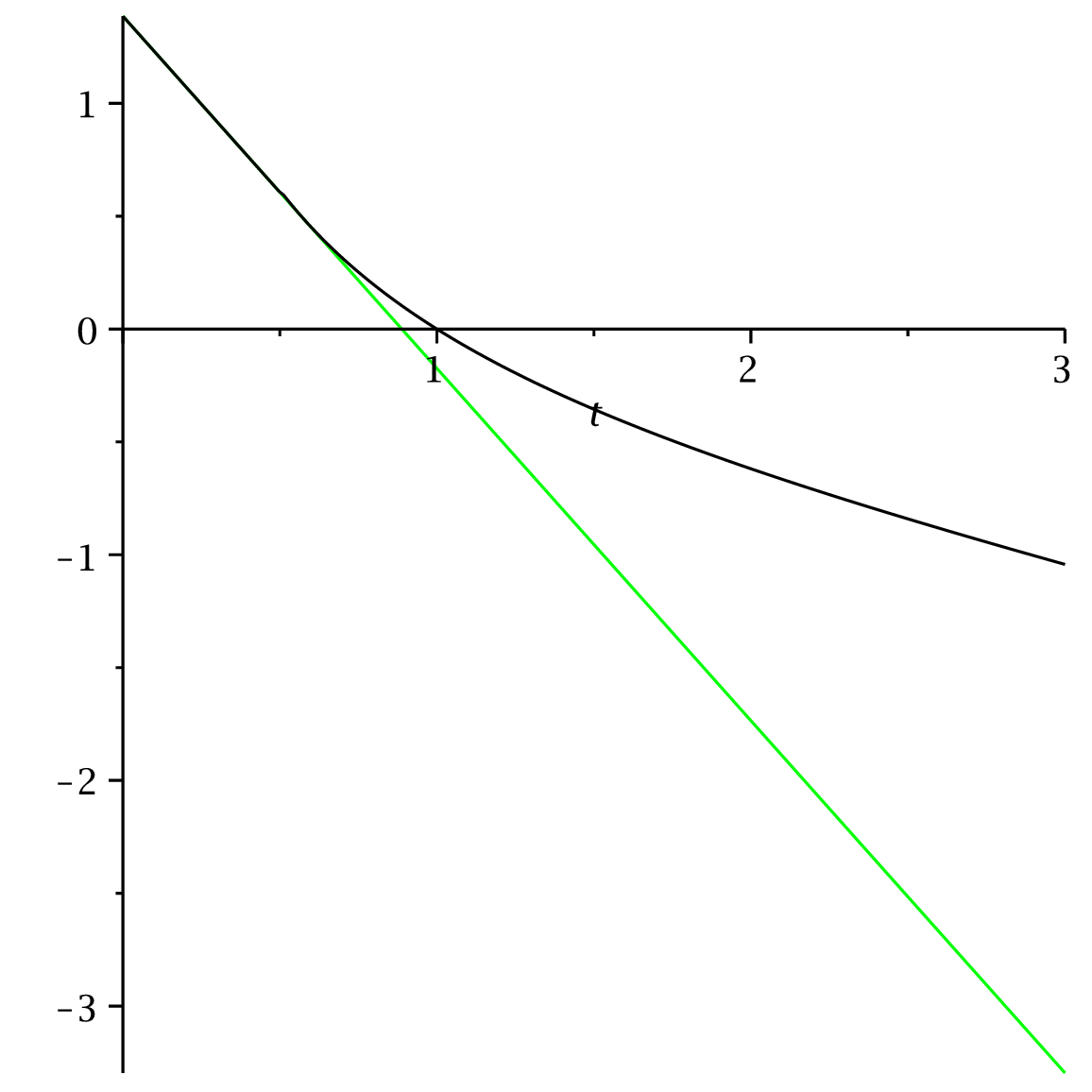}}}
\put(16,1.7)
{\resizebox{5.5cm}{4.5cm}{\includegraphics{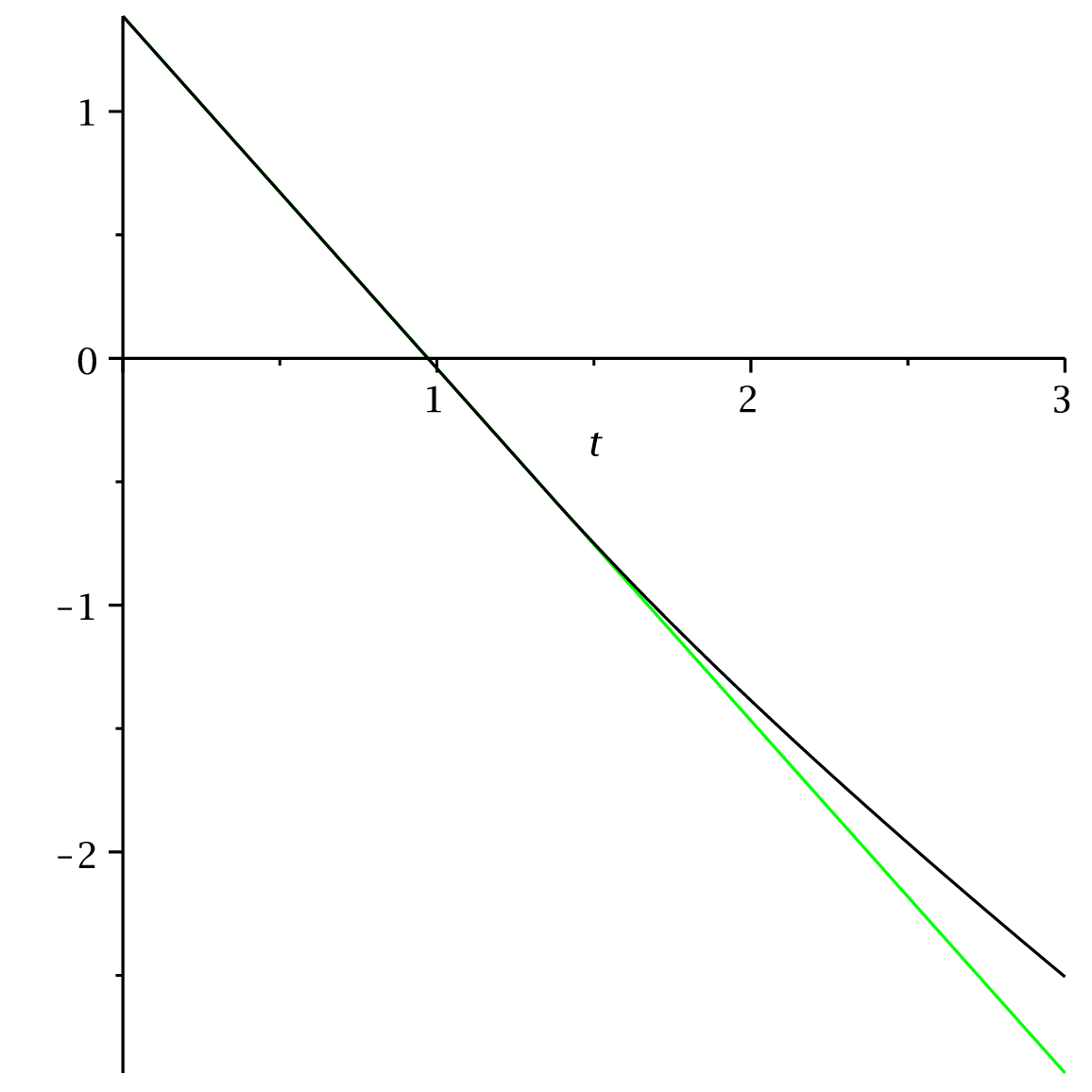}}}
\end{picture}
\end{center}
\caption{\label{fig:x(t)} Graphs of $P(-t \log|F'_\lambda|)$ for
$\lambda = 0.3$ with $t_0 \approx 0.5$, $t_1 = 1$ (left) and $\lambda = 0.6$
with $t_0 \approx 1.4$, $t_1 \approx 0.94$ (right).
In either case, $t \mapsto t \log (1-\lambda)$ is the oblique asymptote on the right. Since $ \frac{1}{\lambda(1-\lambda)} \ge |F'| \ge \frac{1}{1-\lambda}$,
we always have $\log 4 + t \log(1-\lambda) \ge P(\Phi_t) \ge
\log 4 + t \log[\lambda(1-\lambda)]$ (drawn in light colour).
}
\end{figure}

\begin{proof}
For fixed $K \in \N$, write $x_{t,K} = s_{t,K}(1-\lambda)^t
= r_{t,K} \lambda^t(1-\lambda)^t$.
For brevity, we will write $r = r_{t,K}$ and
$r_\infty = \lim_{K \to \infty} r_{t,K}$.

Let $\underline v = (v_1,v_2,\dots )$ be the left eigenvector of $(1-\lambda)^{-t} B^t_K$
for eigenvalue $s_{t,K}$, scaled such that $v_1 = 1$.
By the Perron-Frobenius Theorem, we know that $v_j > 0$ for
all $j$.
The shape of the columns of $B_K^t$ imply that $v_j \le v_{j+1}$ for
all $j \le K$, and in fact $v_{K-1} = v_K$ (as the last two columns
of $B^t_K$ are identical). We can recursively solve
$v_2 = r-\lambda^{-t}$ and
$v_3 = (r-1) v_2 - \lambda^{-t} = r^2-(1+\lambda^{-t})r$, and since
$v_2 > 0$, this already gives $s_{t,K} = \lambda^t r >
\lambda^t(1+\lambda^{-t}) = (1+\lambda^t)$.
Similarly, $v_3 \ge v_2$ implies that $r^2 - (2+\lambda^{-t})r +
\lambda^{-t} \ge 0$,
whence $r \ge \frac12(2+\lambda^{-t} + \sqrt{4+\lambda^{-2t}})$.
The general term is
$$
v_j = (r-1) v_{j-1} - \sum_{k=2}^{j-2} v_k - \lambda^{-t}
\quad \text{ for } 3 \le j \le K.
$$
If we write $v_j = a_j - b_j \lambda^{-t}$,
we find
$$
\left\{ \begin{array}{ll}
a_n = ra_{n-1} - \sum_{k=1}^{n-1} a_k + 1 &\qquad a_2 = r,\ a_1 = 1,\ a_0 = 0
\\
b_n = rb_{n-1} - \sum_{k=1}^{n-1} b_k + 1 &\qquad b_2 = 1,\ b_1 = 0.
\end{array} \right.
$$
An induction proof then gives that
\begin{equation}\label{eq:an}
a_n = r(a_{n-1} - a_{n-2}) \qquad b_n = a_{n-1}.
\end{equation}
This recursive formula has the characteristic equation
$\mu^2-\mu r - r = 0$, with solutions
$\mu_\pm = \frac12 (r \pm \sqrt{r^2-4r})$.
Writing $a_n = A_+ \mu_+^n + A_- \mu_-^n$, we readily find
that $A_+ = -A_- = 1/\sqrt{r^2-4r}$, and therefore
$$
a_n = \left\{ \begin{array}{ll}
\frac{1}{\sqrt{r^2-4r}} \left( \mu_+^n - \mu_-^n\right)
& \text{ if } r > 4;\\[2mm]
n 2^{n-1} & \text{ if } r = 4;\\[2mm]
\frac{2\sin(\alpha n)}{\sqrt{ 4r-r^2} }\ r^{n/2} & \text { if } r < 4
\text{ and } \tan \alpha = \sqrt{\frac4r -1}.
\end{array}\right.
$$
\begin{remark} Putting in the numbers for $a_1 = 1$, $a_2 = r$,
$a_n = r(a_{N-1}-a_{n-2})$, we get
\begin{eqnarray*}
a_3 &=& r^2-r, \\
a_4 &=& r^3 - 2r^2, \\
a_5 &=& r^4-3r^3+r^2, \\
a_6 &=& r^5-4r^4+3r^3, \\
a_7 &=& r^6-5r^5+6r^4-r^3 , \\
a_8 &=& r^7-6r^6+10r^5-4r^4, \\
\vdots\ &  & \qquad \qquad \vdots
\end{eqnarray*}
which are the signed entries along the north-east-east diagonals in
Pascal's triangle, so
$$
a_K(x) = \sum_{j=0}^{\lfloor (K-1)/2 \rfloor} \binom{K+1-j}{j}  (-1)^i r^{K-j}.
$$
\end{remark}
{\bf I}:
First assume that $\lambda^t < \frac12$, so $\lambda^{-t} > 2$.\\
If $r_\infty > 4$, then
$a_j$ will grow exponentially fast with rate
$\mu_+ \ge r/2 \ge 2$, and
$$
v_j = a_j - \lambda^{-t}b_j = \frac{1}{ \sqrt{r^2-4r} }
\left[ (\mu_+-\lambda^{-t} ) \mu_+^{j-1} +
(\lambda^{-t}-\mu_- ) \mu_-^{j-1} \right].
$$
If $\mu_+ > \lambda^{-t}$, then indeed $v_j > 0$ for all $j$,
but we get a contradiction via the  following argument
(which we will call {\em argument A}):
Since $v$ is a left eigenvector for eigenvalue $r_{t,K} = r\lambda^t$,
$$
r v_{K-2} = \sum_{j=2}^{K-1} v_j + v_1 \lambda^{-t}
\quad \text{ and }\quad   rv_{K-1} = \sum_{j=2}^K v_j  + v_1 \lambda^{-t}.
$$
Subtract the two equations, and recall that $v_K = v_{K-1}$.
Then
$r(v_{K-1} - v_{K-2}) =  v_K =  v_{K-1}$,
so $r \sim \mu_+/(\mu_+-1) \le 2$, contrary to the assumption that
$r \ge 4$.

The case $\mu_+ = \lambda^{-t}$ implies that
$$
r_\infty = \frac{\lambda^{-t}}{{1-\lambda^t}}
\quad (\text{whence } x = (1-\lambda)^t \lambda^t r_\infty = \psi(t))
\quad \text{ and } \quad
\mu_- = \frac{1}{1-\lambda^t} > 1.
$$
Note that indeed $r_\infty \ge 4$ if $\lambda^t \le \frac12$.
In this case,
\begin{equation}\label{eq:vj}
v_j =  \frac{\lambda^{-t} - \mu_-}{\sqrt{r^2-4r}}\ \mu_-^{j-1}
= \frac{1}{1-\lambda^t}\ \mu_-^{j-1},
\end{equation}
and argument A gives
$r_\infty =
\mu_-/(\mu_--1) = \lambda^{-t}$, which still looks like a
contradiction. However, argument A relies on having a
finite matrix $B^t_K$ with the two last columns identical.
If $K < \infty$, then we can in fact still take
$\mu_+ < \lambda^{-t}$ close to $\lambda^{-t}$, because \eqref{eq:vj}
only implies that $v_j < 0$ for large $j$.
This gives $x_{t,K} < \psi(t)$, but $\mu_+ \to \lambda^{-t}$ and
$x_{t,K} \to \psi(t)$ as $K \to \infty$.

If $r_\infty = 4$, then $v_j = j2^j - \lambda^{-t} (j-1)2^{j-1}
= 2^j[1 - \frac{\lambda^{-t}-2}{2} (j-1)]$,\ and this is
negative as soon as $j > 1/(1-2\lambda^t)$.

If $r_\infty < 4$, then
\begin{equation}\label{eq:vj2}
v_j = \frac{2}{\sqrt{4r-r^2}}\ r^{(n-1)/2}
\left[ r^{1/2} \sin(\alpha j) -
\lambda^{-t} \sin(\alpha(j-1)) \right],
\end{equation}
and this becomes negative
when $\sin(\alpha j) < \sin(\alpha(j-1))$.

So in the limit $K \to \infty$, this rules out $r \le 4$,
and therefore $x_t = \lim_{K \to \infty} x_{t,K} = \psi(t)$
as claimed.
\\[3mm]
{\bf II}: Assume $\lambda^t = \frac12$, so $\lambda^{-t} = 2$.\\
Now $r_\infty > 4$ gives that $\mu_+ > 2$, so
$v_j = \mu_+^{j-1}(\mu_+ - \lambda^{-t}) + \mu_-^{j-1}(\lambda^{-t}-\mu_-) > 0$; in fact $v_j$ increases exponentially fast
and $v_j/v_{j-1} \to \mu_+$. Therefore argument A implies
that for large $K$, $r \sim \mu_+/(\mu_+-1)
\le r/(r-2) \le r/2$, which is a contradiction.

The case $r_\infty < 4$ leads to a contradiction in the same way
as in Case I.
So the remaining possibility is $r_\infty = 4$.
In this case $v_j = 2^j$, and argument A would
again give a contradiction,
if we could apply it to an infinite matrix.
For finite matrices $B^t_K$, taking $r < 4$
gives $v_j = \frac{2}{\sqrt{4r-r^2}}\ r^{(n-1)/2}
\left[ r^{1/2} \sin(\alpha j) - 2 \sin(\alpha(j-1)) \right]$
as in \eqref{eq:vj2}, but if $r^{1/2}$ is sufficiently close to $2$
and $\alpha = \sqrt{\frac{4}{r}-1}$ sufficiently close to $0$,
then $v_j$ is still positive for $j \le K$.
Hence $r \to 4$ and $x_{t,K} \to 1$ as $K \to \infty$,
\\[3mm]
{\bf III}: The final case is $\lambda^t > \frac12$,
so $\lambda^{-t} < 2$.\\
Now $r_\infty > 4$ fails by argument A as in Case I, and
if $r_\infty < 4$, then $\alpha = \sqrt{\frac{4}{r}-1}$
is bounded away from $0$ uniformly in $K$.
Since $v_j$ is as given by \eqref{eq:vj2}, it becomes negative
for $j$ sufficiently large (and independently of $K$),
namely when $\sin(\alpha j) \le 0 \le \sin(\alpha(j-1))$,

Therefore $r_\infty = 4$, as it is the only way allowing
$\alpha \to 0$ as $K \to \infty$.

This completes the proof \eqref{eq:xt}.
Direct computation shows that at $t = t_0$, the
left and right derivatives are $\log[\lambda(1-\lambda)]$,
but the second left derivative is $0$ and right second derivative
is $2\log^2 \lambda$.
Therefore $t \mapsto \log x(t)$ is $C^1$ but not $C^2$.
The formula for $t_1$ is straightforward.
\end{proof}

\section{The size of hyperbolic and escaping sets}
\label{sec:esc hyp dim}

In this section we compute the Hausdorff dimension
of hyperbolic sets (Proposition~\ref{prop:hypdim}) and
the escaping set (Proposition~\ref{prop:Omega}), which combined
prove Theorem~\ref{thm:main hyp dim}.

\begin{proposition}\label{prop:hypdim}
Let
$\Lambda_K = \{ x \in (0,1] : F^i(x) \notin \cup_{k > K}(W_k) \text{ for all } i\ge 0\}$.
Then
\begin{equation}\label{eq:dimH esc}
\lim_{K \to \infty} \dim_H(\Lambda_K) = \left\{ \begin{array}{ll}
1 & \text{ if } \lambda \le \frac12; \\[1mm]
-\frac{\log 4}{\log[\lambda(1-\lambda)]} \qquad & \text{ if } \lambda \ge \frac12.
\end{array} \right.
\end{equation}
\end{proposition}

\begin{proof}
The cylinder sets $\{ [e_0 \dots e_{n-1}] \}_{e_i \le K}$ form a cover of $\Lambda_K$ of diameter tending to $0$ as $n \to \infty$.
Since equations \eqref{eq:dimcover1} and \eqref{eq:dimcover} hold for the $K \times K$ matrices $A_K^t$ and $B_K^t$ as well, we find
$$
H_{t,K}^n := \sum_{e_i \le K} |[e_0 \dots e_{n-1}]|^t
= \underline w_t \cdot (A^t_K)^{n-1} \cdot \underline 1^T
= \underline 1 \cdot (B^t_K)^{n-1} \cdot \underline w_t^T.
$$
These quantities develop exponentially in $n$ according to the
leading eigenvalue $x_{t,K}$ of $A_K^t$ (or $B^t_K$).
Therefore $\inf\{ t : H_{t,K}^n < \infty \text{ for all } n \in \N\}$
coincides with the first zero of $t \mapsto \log x_{t,K}$.
Taking the limit $K \to \infty$, we find by Theorem~\ref{thm:B ev} that
$$
\lim_{K\to \infty}\dim_H(\Lambda_K) \le t_1 =
\left\{ \begin{array}{ll} 1  & \text{ if } \lambda \le \frac12; \\[1mm]
-\frac{\log 4}{\log[\lambda(1-\lambda)]}  & \text{ if } \lambda \ge \frac12. \end{array} \right.
$$
Now for a lower bound, we first treat the case $\lambda \geq \frac12$. Take  $\hat \Lambda_{K+1} = \{ x \in  \Lambda_{K+1} :
F_\lambda(x) \notin W_1 \text{ for all } n \ge 0\}$.
Observe that
$\{ [e_0 \dots e_{n-1}] \}_{2 \le e_i \le K+1}$ is a cover of
$\hat \Lambda_{K+1}$ of the same cardinality as
$\{ [e_0 \dots e_{n-1}] \}_{1 \le e_i \le K}$,
and consisting of intervals of length
$[\lambda(1-\lambda)]^{n-1} |W_{e_{n-1}}|$.
Moreover, $\hat \Lambda_{K+1}$ is a self-similar Cantor set
(with bounded ratios between bridges and gaps)
and by fairly standard arguments one can conclude that
its dimension is given by the zero of the leading eigenvalue
$t \mapsto \log \hat x_{t,K+1}$  of $\hat A_{K+1}^t$, which is defined
as $A_{K+1}^t$ with the first row and column removed.
The same argument as used in Lemma~\ref{lem:charA=charB}
shows that  $\hat A_{K+1}^t$ and  $\hat B_{K+1}^t$ (which is defined
as $B_{K+1}^t$ with the first row and column removed)
have the same characteristic polynomial, but
$$
\hat B_{K+1}^t = \lambda^t(1-\lambda)^t
\left( \begin{array}{ccccccc}
1 & 1 & 1 & 1 & 1 & \hdots   & \hdots \\
1 & 1 & 1 & 1 & 1 & \hdots   & \hdots \\
0 & 1 & 1 & 1 & 1 & 1 & \hdots \\
0 & 0 & 1 & 1 & 1 & 1 & \hdots \\
\vdots & \vdots & 0  & 1 & 1 & 1 & \hdots \\
\vdots & \vdots & \vdots  & \vdots   & \vdots   & \vdots   & \ddots
\end{array} \right) = \lambda^t (1-\lambda)^t B^0_K
$$ 
as in Theorem~\ref{thm:B ev}. 
Therefore the leading eigenvalue $\hat x_{t,K} = \lambda^t(1-\lambda)^t x_{0,K} \to 4 \lambda^t(1-\lambda)^t$
as $K \to \infty$.
The required lower bound for the zero $t_1$ follows,
proving \eqref{eq:dimH esc}.

Now for the case $\lambda < \frac12$, 
write $F=F_\lambda$ (we will use $\lambda$ for leading eigenvalue shortly). 
Then as in Theorem~\ref{thm:RPF}, $P(-\log|DF|)=0$.  Moreover, $P(-t\log|DF|)>0$ for any $t<1$.  
Truncate the system at symbol $K$ to get pressure $P_K$ and note that by Sarig's theory \cite{Sartherm} 
for $K$ large enough $P_K(-t\log|DF|)>0$.  There must be an equilibrium state $\mu$ on this 
truncated system for this potential and it has
$$
\dim_H \mu =\frac{h(\mu)}{\lambda(\mu)} = \frac{P_K(-t\log|DF|)}{\lambda(\mu)} +t>t,
$$
which implies $\dim_H(\Lambda_K)>t$ also.  Recalling that $t$ was an arbitrary number $<1$, we obtain the lower bound $1$.
\end{proof}

The following proposition, giving a different proof to the
result of Stratmann \& Vogt
\cite{StraVogt},
proves the other part of Theorem~\ref{thm:main hyp dim}.

\begin{proposition}\label{prop:Omega}
The dimension $\dim_H(\Omega_{\lambda})$ is given by
\begin{equation}\label{eq:dimH_Omega}
\dim(\Omega_\lambda) = \left\{ \begin{array}{ll}
-\log 4/\log[\lambda(1-\lambda)] \qquad & \text{ if } \lambda \le \frac12; \\[1mm]
1 & \text{ if } \lambda \ge \frac12.
\end{array} \right.
\end{equation}
\end{proposition}

\begin{proof}
Let $\Omega'_{\lambda} = \{ x \in \Omega_{\lambda} : F_\lambda^n(x) \notin W_1 \text{ for all } n \ge 0\}$.
Since $\Omega_{\lambda} = \cup_n F_\lambda^{-n}(\Omega'_{\lambda})$,
it suffices to compute the Hausdorff dimension of
$\Omega'_{\lambda}$.
We can use \eqref{eq:dimcover} to approximate the $t$-dimensional
Hausdorff mass of $\Omega'_{\lambda}$, and if we replace the first
row of $A^t = A^t_{\lambda}$ by zeroes, this has no effect on the
estimate, because by definition no point in $\Omega'_{\lambda}$
 ever ``uses'' the first state.

For $\lambda > \frac12$ we know that $F_\lambda$ is dissipative,
so $\lim_{n\to \infty}F_\lambda^n(x) = 0$ Lebesgue-a.e.\ $x$, and the Hausdorff
dimension of such points is certainty $1$, proving the proposition in this case.
Therefore take $\lambda < \frac12$.
Let $\eta \in (0, \frac{\log 2}{\log \lambda} \log \frac{\lambda}{1-\lambda} )$
be small and $\gamma \in (\frac12, 1-\lambda)$.
Since $\gamma$ corresponds to dissipative behaviour,
the Lebesgue measure of $\Omega_\gamma'$ is positive,
so its Hausdorff dimension is $1$.
Recall from \eqref{eq:drift} that the drift
$\drift(\gamma) = \frac{2\gamma-1}{1-\gamma}$, which tends to $0$ as $\gamma \to \frac12$.
Recall that $\chi_n(x) = k$ if $F_\lambda^n(x) \in W_k$.
Measured w.r.t.\ Lebesgue measure, the vast majority of points
satisfy $\chi_n(x) \sim n \drift(\gamma)$, so that in fact the set of points
$\Omega''_{\gamma} = \{ x \in \Omega'_\gamma : \chi_n(x) \le 10\drift(\gamma)n \text{ for all $n$ sufficiently large}\}$ has positive Lebesgue measure.

Note that the sets $\Omega$, $\Omega'$ and $\Omega''$ depend on the parameter
$\lambda$ (or $\gamma$), but the codes, or equivalently the
sequences $\chi_n(x)$ for $x$ in these sets are independent of the parameter.

Since $\dim_H(\Omega''_\gamma) = 1$, for each $u < 1$,
we can find a mesh $\eps$ covers ${\mathcal U}_\eps$ of $\Omega_\gamma$
using cylinder sets
of variable lengths, such that $\sum_{C \in {\mathcal U}_\eps} |C|^u$
diverges as $\eps \to 0$.

The idea for $\lambda < \frac12$ is now to use a cover of cylinders that for
$\gamma$ represents positive measure, and hence
finite $u$-dimensional Hausdorff mass for all $u < 1$.
Choose $t =  \frac{ u \log[\gamma(1-\gamma)] + \eta \log \lambda}{ \log[\lambda(1-\lambda]} \in (0,1)$, so
$$
[\lambda(1-\lambda)]^t = [\gamma(1-\gamma)]^u e^{u\eta}.
$$
Because
$\eta < \frac{\log 2}{\log \lambda} \log \frac{\lambda}{1-\lambda}$
we obtain for
and $\gamma$ sufficiently close to $\frac12$ that
$$
\log \frac{ \lambda^t }{ \gamma^u } = t \log \lambda - u \log \gamma =
u\ \frac{\log (1-\gamma)\log \lambda -  \log\gamma \log(1-\lambda) + \eta }{ \log[\lambda(1-\lambda)] } < 0,
$$
so $\lambda^t < \gamma^u$.
We use the same covers ${\mathcal U}_{\eps}$ and note that
$t$-conformal length of cylinders $C$ for parameter $\lambda$
coincides with the $u$-conformal length of $C$
for parameter $\gamma$.
Hence, indicating the parameter used in computing the length
of intervals by a subscript, we get by \eqref{eq:w}
that for an $n$-cylinder $C$ such that $F_\lambda^{n-1}(C) = W_{\chi_n(C)}$:
$$
|C|_\lambda^t = [\lambda(1-\lambda)]^{t(n-1)} m_t(W_{\chi_n(C)})
= [\lambda(1-\lambda)]^{tn} \lambda^{t(\chi_n(C)-2)}
$$
and similar for $|C|_\gamma^u$.
Summing over all such cylinders, we get
$$
\sum_{C \in {\mathcal U}_\eps} |C|_\lambda^t =
\sum_{C \in {\mathcal U}_\eps} |C|_\gamma^u\ e^{u\eta n(C)}\ \left( \frac{\lambda^t}{\gamma^u} \right)^{\chi_n(C)-2}.
$$
By definition of $\Omega''_\gamma$, we have $\chi_n(C) \le 10 \drift(\gamma)n$
and thus $\left( \frac{ \lambda^t }{ \gamma^u } \right)^{\chi_n(C)} \ge e^{-u \eta n}$
provided $\gamma$ is sufficiently close to $\frac12$.
Therefore the right hand side in the formula diverges as
 $\gamma \searrow \frac12$ and $\eps \to 0$.
 Furthermore, the mesh size of ${\mathcal U}_\eps$
is different for parameter $\lambda$ and $\gamma$, but they tend to $0$
for both parameters simultaneously as $\eps \to 0$.
Therefore $t \le \dim_H(\Omega_\lambda)$.
Taking the limits
$\gamma \searrow \frac12$, $u \nearrow 1$ and $\eta \to 0$, we obtain the required
lower bound
$\dim_H( \Omega_{\lambda}) \ge
\frac{- \log 4}{\log\lambda(1-\lambda)}$.

For the upper bound, we take $u > 1$ and $\eta = 0$.
Then the $u$-dimensional Hausdorff mass of
$\eps$-covers ${\mathcal U}_\eps$ converges for parameter $\gamma = \frac12$.
So now, taking $\eta = 0$ and the limit $u \searrow 1$,
gives the same upper bound $\dim_H(\Omega_{\lambda}) \le
\frac{- \log 4}{ \log\lambda(1-\lambda)}$.
This completes the proof for $\lambda < \frac12$.
Finally, monotonicity of $\lambda \mapsto \dim_H(\Omega_\lambda)$
gives $\dim_H(\Omega_\lambda) = 1$ for $\lambda = \frac12$.
\end{proof}

\begin{remark}
Notice the striking symmetry:
$\dimhyp(F_{\lambda}) = \dim_H(\Omega_{1-\lambda})$.
We can explain this using an argument from \cite{StraVogt}, namely
a coding of the system $((0,1], F_\lambda)$ based on ``$\lambda$-adic''
partitions that are defined inductively by starting with $[0,1]$, and
dividing all intervals of the previous stage into two parts of relative lengths
$1-\lambda$ (with symbol $0$) and $\lambda$ (with symbol $1$).
This means that Lebesgue measure on $[0,1]$ corresponds to $(1-\lambda, \lambda)$
Bernoulli measure on $\{ 0, 1\}^{\N_0}$.

As a result, if $x$ has code
$0^{n_0}10^{n_1}10^{n_2}10^{n_3}1\dots$ then
$x \in W_{n_1}$,
$$
F_\lambda(x) \in \left\{ \begin{array}{ll}
W_{n_0+n_1-1} & \text{ if } x \notin W_1, \\
W_{n_1}  & \text{ if } x \in W_1.
\end{array} \right.
$$
and in general,
$$
F_\lambda^k(x) \in \left\{ \begin{array}{ll}
W_{j+n_k-1} & \text{ if } F_\lambda^{k-1}(x) \in W_j, j \ge 2 \\[2mm]
W_{n_k}  & \text{ if } F_\lambda^{k-1}(x) \in W_1.
\end{array} \right.
$$
Let $Y$ be the set of point such that $0s$ dominate in their codes,
\ie
$$
Y = \left\{ y \in [0, 1] : \lim_{r \to \infty} \sum_{j=0}^{r-1} n_j(y) - r \to \infty\right\}.
$$
These are precisely the points such that $F_\lambda^r(y) \to 0$, and in fact
$F_\lambda^r(y) \in \cup_{i \ge j} W_i$
if $j = \sum_{j=0}^{r-1} n_j - r $ with equality if $F_\lambda^j(y)$ does not linger
in $W_0$ for successive iterates $j \le r$.
Stratmann \& Vogt show that $\dim_H(Y)$ is given by \eqref{eq:dimH_Omega}.

Let us now form $\hat Y$ as the set of points $\hat y$ with opposite
codes as $Y$, \ie $\hat y$ is the point obtained by changing all $0$s
in the code of $y$ by $1$s and vice versa, and let us also
change $\lambda$ to $\hat \lambda = 1-\lambda$.
Then $\dim_H(\hat Y_{\hat \lambda}) = \dim_H(Y_\lambda)$.
But $\hat Y$ are points in whose code $1$s dominate, so
their orbits visit only finitely many intervals $W_j$,
and hence $\dim(\hat Y_{\hat \lambda}) = \dimhyp(F_{\hat \lambda})$,
which explains the symmetry.

The only exception for this argument
are point $\hat y$ that remain in $W_0$ for a long time $n_k$
(a block of $n_k$ ones in the code) and then visit $W_{j_k}$
for $1 \ll j_k \ll n_k$ (a block of $j_k$ zeroes in the code).
The regularity of such codes makes is plausible that the Hausdorff dimension
of such points is small and hence has no effect on the equality
 $\dim(\hat Y_{\hat \lambda}) = \dimhyp(F_{\hat \lambda})$.
\end{remark}

\section{Topological and Gurevich Pressure}
\label{sec:pressure}

In this section we present the classical definition (see
\cite{Ru_press, Walt_press, Waltbook})
of topological pressure along with a Gurevich definition of pressure
for countable Markov graphs
which allows us to prove Theorem~\ref{thm:main pressure formula} and
Corollary~\ref{cor:main all press}.
The results here also set the stage for the proof of the null
recurrent case in Section~\ref{sec:null rec}.

Let $f:X \to X$ be a continuous map on a metric space.
Following Bowen \cite{Bo_balls}, let
$$
d_n(x, y):=\max \{d(f^k(x), f^k(y)):0 \le k < n\}
$$
Given $\eps>0$ we say that a finite set $E\subset X$ is \emph{$(n, \eps)$-separated} if $d_n(x, y)>\eps$ for every $x,y\in E$ such that $x\neq y$.
Bowen showed that topological entropy coincides with
the exponential growth rate in $n$ of the
maximal cardinality of $(n, \eps)$-separated sets, but in order to
obtain pressure, one needs to compute ergodic sums of
of the potential on each point in an $(n, \eps)$-separated set.
Let $E_{n,\eps}$ be the collection of all $(n, \eps)$-separated sets.
Define
\begin{equation}\label{eq:Gamma}
\Gamma_{n,\eps}(\phi):=\sup_{E_{n,\eps}}\sum_{x\in E_{n, \eps}} e^{S_n\phi(x)},
\end{equation}
where $S_n\phi(x):=\phi(x)+\cdots +\phi\circ\sigma^{n-1}(x)$.
The classical topological pressure introduced in \cite{Ru_press, Walt_press} is
$$
\ptop(\phi):=\lim_{\eps\to 0}\limsup_{n\to\infty}\frac1n \log \Gamma_{n,\eps}.
$$
Of course, our maps $F_\lambda:(0, 1] \to (0,1]$ are not continuous
as interval maps.  However, we can still compute topological pressure for them.

\begin{remark}
In the compact setting,
since all metrics generating the same topology are uniformly equivalent
($d_1$ and $d_2$ are called {\em uniformly equivalent} if the
identity maps from $(X,d_1)$ to $(X, d_2)$ and vice versa are both uniformly
continuous), the value of the pressure does not depend upon the metric (see \cite[Section 7.2]{Waltbook}). However, in non-compact settings this may no longer be the case.  This is one of the motivations for the alternative notion of pressure given in the next subsection, which in our situation is shown to agree with $\ptop(\phi)$.
\end{remark}

Since $F_\lambda$ preserves
the countable Markov partition $\{W_k\}_{k \in \N}$ it is natural to
use a countable Markov shift (CMS) on alphabet $\N$.
By the definition of $W_n$ as half-open intervals, every point
(rather than almost every) has a well-defined symbolic itinerary,
no information is lost when passing from the interval to symbolic
dynamics.
With the theory we present here we can interpret some of the results proved above about eigenvalues of matrices in terms of the pressure.

Let $\sigma \colon \Sigma \to \Sigma$ be a one-sided Markov shift
with a countable alphabet $\N$. That is, there exists a matrix
$(t_{ij})_{\N \times \N}$ of zeros and ones (with no row and no column
made entirely of zeros) such that
\[
\Sigma=\{ x\in \N^{\N_0} : t_{x_{i} x_{i+1}}=1 \ \text{for every $i
\in \N_0$}\},
\]
and the shift map is defined by $\sigma(x_0x_1 \cdots)=(x_1 x_2
\cdots)$. We say that $(\Sigma,\sigma)$ is a \emph{countable
Markov shift}.
We equip $\Sigma$ with the topology generated by the cylinder sets
$$
[e_0 \cdots e_{n-1}] = \{x \in \Sigma : x_j = e_j \text{ for } 0 \le j < n \}.
$$
By making the move from the interval to the coding space $\Sigma$ we lose
connectedness, but gain smoothness for our potentials.

Given a function
$\phi\colon \Sigma \to\R$, for each $n \ge 1$ we define the \emph{variation}
on $n$-cylinders
\[
V_{n}(\phi) = \sup \left\{|\phi(x)-\phi(y)| : x,y \in \Sigma,\
x_{i}=y_{i} \text{ for } 0 \le i < n \right\}.
\]
We say that $\phi$ has \emph{summable variations} if
$\sum_{n=2}^{\infty} V_n(\phi)<\infty$.  We will sometimes refer to $\sum_{n=2}^{\infty} V_n(\phi)$ as the \emph{distortion bound} for $\phi$.  Clearly, if $\phi$ has
summable variations then it is continuous.   We say that $\phi$ is \emph{weakly H\"older continuous} if $V_n(\phi)$ decays exponentially.  If this is the case then it has summable variations. In what follows we assume
$(\Sigma, \sigma)$ to be topologically mixing (see \cite[Section 2]{Sartherm} for a precise definition).

Based on work of Gurevich \cite{Gutopent, Gushiftent}, Sarig \cite{Sartherm} introduced a notion of pressure for countable Markov shifts which does not depend upon the metric of the space and which satisfies a Variational Principle.
Let $(\Sigma, \sigma)$ be a topologically mixing countable Markov shift, fix a symbol $e_0$ in the alphabet $S$ and
let $\phi \colon \Sigma \to \R$ be a potential of summable variations.  
We let the {\em local partition function at $[e_0]$} be 
\begin{equation}
Z_n(\phi, [e_0]):=\sum_{x:\sigma^{n}x=x} e^{S_n\phi(x)} \chi_{[e_{0}]}(x)
\label{eq:Zn}
\end{equation}
and
$$
Z_n^*(\phi, [e_0]) := \sum_{\stackrel{x:\sigma^{n}x=x,}{x:\sigma^{k}x \notin [e_{0}] \ \mbox{\tiny for}\ 0< k < n}} \hspace{-10mm} e^{S_n\phi(x)} \chi_{[e_{0}]}(x),
$$
where $\chi_{[e_{0}]}$ is the characteristic function of the
$1$-cylinder $[e_{0}] \subset \Sigma$,
and $S_n\phi(x)$ is $\phi(x) + \dots + \phi \circ \sigma^{n-1}(x)$.
The so-called \emph{Gurevich pressure} of $\phi$ is defined
by
the exponential growth rate
\[
 P_G(\phi) := \lim_{n \to
\infty} \frac{1}{n} \log Z_n(\phi, [e_0]).
\]
Since $\sigma$ is topologically mixing, one can show that $P_G(\phi)$ does not depend on $e_0$.  If $(\Sigma, \sigma)$ is the full-shift on a countable alphabet then the Gurevich pressure coincides with the notion of pressure introduced by Mauldin \& Urba\'nski \cite{MUifs}.

We defined transience/recurrence of a system in the introduction in terms of the relevant measures there. In the CMS context, as proved in \cite{Sarnull}, these are equivalent to the following definitions.
The potential $\phi$ is said to be {\em recurrent} if\footnote{The convergence of this series is independent of the cylinder set $[e_0]$,
so we suppress it in the notation.}
\begin{equation}\label{eq:recurrence}
\sum_n e^{-nP_G(\phi)} Z_n(\phi) = \infty.
\end{equation}
Otherwise $\phi$ is \emph{transient}.
Moreover, $\phi$ is called \emph{positive recurrent} if it is
recurrent and $$\sum_n ne^{-nP_G(\phi)}Z^*_n(\phi) < \infty.$$

The following can be shown using the proof of \cite[Theorem 3]{Sartherm}.

\begin{proposition}[Variational Principle] \label{prop:VarPri}
If $(\Sigma, \sigma)$ is topologically mixing and  $\phi: \Sigma \to \mathbb{R}$ has summable variations, $\phi<\infty$ and $\phi$ is weakly H\"older continuous, then
\begin{equation*}
P_G(\phi)= P(\phi). 
\end{equation*}
\end{proposition}
Let us stress that $P(\phi)$  only
depends on the Borel structure of the space and not on the
metric . Therefore, $P_G(\phi)$ must also be independent of the metric on the space.

The Gurevich pressure also has the property that it can be approximated by its restriction to compact sets. More precisely
\cite[Corollary 1]{Sartherm}:

\begin{proposition}[Approximation Property] \label{prop:approx}
If $(\Sigma, \sigma)$ is topologically mixing and $\phi: \Sigma \to \mathbb{R}$ is weakly H\"older continuous then
\begin{equation*}
P_G( \phi) = \sup \{ P_{top|K}( \phi) : \emptyset \ne K \subset \Sigma, K \text{ is compact and shift-invariant}  \},
\end{equation*}
where $P_{top|K}( \phi)$ is the topological pressure on $K$.
\end{proposition}

At this point we can prove that for our systems, $\ptop$ and $P_G$ coincide.  This result is similar to \cite[Proposition 1.3]{Salama}

\begin{proposition}
For $\lambda\in (0, 1)$ and $t\ge 0$,
$\ptop(-t\log|F_\lambda'|)=P_G(-t\log|F_\lambda'|)$.
\label{prop:class eq gur}
\end{proposition}

\begin{proof}
For a subset $K\subset X$, let $\Gamma_{n, \eps}(\phi)|_K$ be the above
quantity restricted to the set $K$.
We first claim that for our map $F$ and for each $\eps>0$,
there exists $K_\eps\subset I$ such that
\begin{equation}\label{eq:Gamma ineq}
\Gamma_{n,\eps}\le n\Gamma_{n,\eps}|_{K_\eps}.
\end{equation}
Indeed, for $\eps>0$ there exists a minimal $n(\eps)\ge 1$ such that
$\left|\cup_{n\ge n(\eps)}A_n\right|<\eps$.
Set $K_\eps:=I\sm\left(\cup_{n\ge n(\eps)}A_n\right)$.  Then by the structure
of $F$, there is just one element from $I\sm K_\eps$
entering $K_\eps$ at each successive iterate of $F$.  These new contributions
(which are initially of weight $(\lambda(1-\lambda))^{nt}$,
but eventually can be of the form $(1-\lambda)^{nt}\lambda^k$) can be
paired with a summand already in the sum for $\Gamma_n|_{K_\eps}$.
The number of these new terms generated up to time $n$ is $\le n$ so
\eqref{eq:Gamma ineq} is in fact a big over-estimate.

The Variational Principle for finite shifts on compact sets implies
$\ptop(-t\log|F'|_{K_\eps}|)=P_G(-t\log|F'|_{K_\eps}|)$.
Since \eqref{eq:Gamma ineq} implies
$$
\lim_{\eps\to 0}P_G(-t\log|F'|_{K_\eps}|)\to \ptop(-t\log|F'|),
$$
adding this to Proposition~\ref{prop:approx}
gives $\ptop(-t\log|F'|)=P_{G}(-t\log|F'|)$, as required.
\end{proof}

We use the standard {\em transfer operator}
$(L_\phi v)(x) = \sum_{\sigma y =x} e^{\phi(y)} v(y)$,
with dual operator $L^*_\phi$.
Notice that a measure $m$ is $\phi$-conformal if and only if
$L^*_\phi m = m$.

The following theorem is  \cite[Theorem 1]{Sarnull}.

\begin{theorem}
Suppose that $(\Sigma, \sigma)$ is topologically mixing, $\phi: \Sigma \to \mathbb{R}$ has summable variations and $P_G(\phi)<\infty$.  Then $\phi$ is recurrent if and only if there exists $\lambda>0$ and a conservative sigma-finite measure $m_\phi$ finite and positive on cylinders, and a positive continuous function $h_\phi$ such that $L_\phi^*m_\phi=\lambda m_\phi$ and $L_\phi h_\phi=\lambda h_\phi$.  In this case $\lambda=e^{P_G(\phi)}$.  Moreover,
\begin{enumerate}
\item if $\phi$ is positive recurrent then $\int h_\phi~dm_\phi<\infty$;
\item if $\phi$ is null recurrent then $\int h_\phi~dm_\phi=\infty$.
\end{enumerate}
\label{thm:RPF}
\end{theorem}

Moreover the next theorem follows by \cite[Theorem 2]{Sarnull}:

\begin{theorem}
Suppose that $(\Sigma, \sigma)$ is topologically mixing and $\phi: \Sigma \to \mathbb{R}$ is weakly H\"older continuous and positive recurrent.  Then for the measure $d\mu=h_\phi dm_\phi$ given by Theorem~\ref{thm:RPF}, if $-\int\phi~d\mu$ is finite then $\mu$ is an equilibrium state for $\phi$.
\label{thm:eq}
\end{theorem}

\begin{proposition}
Suppose that $(\Sigma, \sigma)$ is topologically mixing and $\phi: \Sigma \to \mathbb{R}$ is H\"older continuous, has finite Gurevich pressure and is transient or null recurrent.  Then there is no equilibrium state for $\phi$.
\label{prop:no eq}
\end{proposition}

\begin{proof}
We may assume that $\phi$ has $P(\phi)=0$, otherwise we can shift by $P(\phi)$. Now we use an inducing argument.  We fix a state $a\in S$, and derive the induced system $(\overline{X}, \overline{\sigma}, \overline{\phi})$ as a first return map to $a$.  This is the full shift on countably many symbols, which, as shown in \cite{muGIBBS} and \cite{SarBIP} has many strong properties.  We will use these to guarantee that we have no equilibrium state in the non-positive recurrent case.

We begin by noting that \cite{muGIBBS} and \cite{SarBIP}, since $(\overline{X}, \overline{\sigma})$ is the full shift, $\overline\phi$ is necessarily positive recurrent whenever $P(\overline\phi)<\infty$ for any choice of $a\in S$.
We can in fact show that in all cases $P(\overline\phi) \le 0$. Indeed, if $P(\overline\phi)> 0$ then by Proposition~\ref{prop:approx} we can take a compact invariant subset of $(\overline X, \overline \phi)$ which still has strictly positive pressure and a corresponding equilibrium state $\overline\nu$.  By the Abramov formula for the projection $\nu$ of $\overline\nu$ to $\Sigma$, we have $h(\nu)+\int\phi~d\nu>0$ contradicting Proposition~\ref{prop:VarPri}.

Now if there is an equilibrium state $\mu$ (hence with unit mass)
for $\phi$ then $h(\mu)+\int\phi~d\mu=0$.  Let $a$ be a state which is given positive mass by $\mu$ and let $\overline\mu$ be the rescaled measure on $[a]$.  Then the Abramov formula implies that $h(\overline\mu)+\int\overline\phi~d\overline\mu=0$ and so Proposition~\ref{prop:VarPri} implies that $P(\overline\phi)\ge 0$.  Thus $P(\overline\phi)= 0$.  We now apply \cite[Corollary 2]{Sartherm}, which when added to part (2) of \cite[Corollary 2]{SarBIP}, says that any equilibrium state for $\overline\phi$ must be of the form obtained in Theorem~\ref{thm:eq}.
Thus $d\overline\mu=h_{\overline\phi}dm_{\overline\phi}$ where $L_{\overline\phi}h_{\overline\phi}=h_{\overline\phi}$ and $L_{\overline\phi}^*m_{\overline\phi}=m_{\overline\phi}$. The functional form of the Kac's Lemma, shown in \cite[Lemma 3]{Sarphase} implies that $\mu$ must also be of this form (\ie $d\mu= h_{\phi}dm_{\phi}$ where
$L_{\phi}h_{\phi}=h_{\phi}$ and  $L_{\phi}^*m_{\phi}=m_{\phi}$), which by
Theorem~\ref{thm:RPF} contradicts the assumption that $\mu$ was not positive recurrent.
\end{proof}

We can now use the theory for thermodynamic formalism for countable Markov shifts to prove Theorem~\ref{thm:main pressure formula} and Corollary~\ref{cor:main all press}.  These follow almost immediately from the results presented in this and in previous sections since, as shown below, the transfer operator can be interpreted in terms of the matrix $B^t$.

\begin{proof}[Proof of Theorem~\ref{thm:main pressure formula}]
We start by clarifying the link between transfer operators for simple potentials and their matrix representations.  Let $(\Sigma, \sigma)$ be a CMS where, for simplicity, we take $\Sigma=\N^{\N_0}$.  Then given a potential $\phi:\Sigma \to \R$ which only depends on one coordinate (\eg $V_1(\phi)=0$), one can form the corresponding infinite matrix $D=D_\phi=(d_{i,j})_{i,j \in \N}$ as $d_{i, j}=\phi(i)$ for all $i\in \N$.  Now for a function $\xi:\Sigma\to \R$ which depends only on one coordinate, we define $\underline\xi$ to be the vector $(\xi(1), \xi(2),\ldots)$, and $\underline{e}_i$ to be the row vector with all zeros except in the $i$-th entry, which is 1.  Then we can compute that for any $x$ in the $1$-cylinder $[i]$,
\begin{equation}
(L_{\phi}\xi)(x)=\underline{e}_i(\underline{\xi}\cdot D) \text{ and } (L_{\phi}^*\xi)([i])= (D \cdot\underline{\xi}^T)\underline{e}_i.
\label{eq:trans mat}
\end{equation}
Thus the leading eigenvalue of the matrix $D$ is the exponential of the
Gurevich pressure of $\phi$.

The fact that the leading eigenvalue of $B^t$ is the exponential of the pressure follows from \eqref{eq:trans mat} and thus the expression for $P(\Phi_t)$ follows from Theorem~\ref{thm:B ev}. The fact that the pressure function is not $C^2$ at $t_0=\frac{-\log2}{\log\lambda}$ follows since by Lemma~\ref{lem:psi}, $D^2\psi(t_0)>0$.

The existence of $\mu_t$ when $\lambda^t<1/2$ follows from Proposition~\ref{prop:mut}.  Uniqueness follows from Theorem~\ref{thm:RPF}. The fact that $\mu_t$ is an equilibrium state follows from Theorem~\ref{thm:eq}.  The non-existence of an equilibrium state when $\lambda^t\ge 1/2$ follows from
Proposition~\ref{prop:no eq}.
\end{proof}

We finish this section with the proof of Corollary~\ref{cor:main all press}.

\begin{proof}[Proof of Corollary~\ref{cor:main all press}]
Let us recall equation \eqref{eq:press eq}:
$$
\begin{array}{l}
P(\Phi_t) = P_G(\Phi_t)=\ptop(\Phi_t)=\Pconf (\Phi_t) = \log \sigma (B^t) \\[3mm]
\qquad \quad  = \lim_{K\to \infty}\log \sigma (B_K^t) =
\lim_{K\to \infty}\log \sigma (A_K^t).
\end{array}
$$
The first equality follows by Proposition~\ref{prop:VarPri}.  The second follows by Proposition~\ref{prop:class eq gur}.  The third and fourth follow by Theorem~\ref{thm:RPF}.  The fifth follows by Proposition~\ref{prop:approx} and the sixth follows by Lemma~\ref{lem:charA=charB}.
If $t = 0$, Theorem~\ref{thm:B ev} gives $\lim_{K\to \infty}\log \sigma (B_K) =
\lim_{K\to \infty}\log \sigma (A_K) = \log 4$.
\end{proof}

\section{Null recurrent case}
\label{sec:null rec}

\begin{lemma}
If $\lambda^t=1/2$ then $((0,1], F_\lambda, \Phi_t)$ is null recurrent.
\end{lemma}

\begin{proof}
Since $Z_k(\Phi_t,A_1)= \underline1\cdot D^{k-1}\cdot(1, 0, 0,
\ldots)^T$ for
$$
D := \left( \begin{array}{ccccccc}
\frac12      & \frac12 & \frac12 & \frac12 & \frac12 & \hdots   & \hdots \\
\frac14      & \frac14 & \frac14 & \frac14 & \frac14 & \hdots   & \hdots \\
0      & \frac14      & \frac14   & \frac14 & \frac14 & \frac14 & \hdots \\
0      & 0      & \frac14        & \frac14   & \frac14 & \frac14 & \hdots \\
\vdots & \vdots & 0        & \frac14        & \frac14   & \frac14 & \hdots \\
\vdots & \vdots & \vdots   & \vdots   & \vdots   & \vdots   & \ddots
\end{array} \right),
$$
the lemma can be proved by determining the form of the first column of
the matrices $D^{n-1}$ (although, of course, we only really care about
the term in the top left corner).  Note that the leading eigenvalue of
this matrix is $1$,
so a priori, the terms of interest could decrease at any subexponential rate.

\begin{claim}
For $k\ge 1$, we have
\begin{equation*}
(D^k)_{i, 1}=\begin{cases} p_{k, k-i+1}/2^{2k-1} & \text{ if } i=1\\
p_{k, k-i+2}/2^{2k} & \text{ if } 2\le i\le k+1\\
0 & \text{ if } i> k+1
\end{cases},
\end{equation*}
for binomial coefficients $p_{k,i}:=\binom{k+2(i-1)}{2(i-1)}$.
\end{claim}

\begin{proof}
We denote the first column of $D^n$ by $\underline{v}^n$.  The columns $\underline{v}^1, \ldots, \underline{v}^5$ are:
$$
\left( \begin{array}{c} \frac12 \\ \frac1{2^2}\\ 0\\ 0\\ 0\\ 0\\ 0\\ \vdots \end{array}\right),
\left( \begin{array}{c} \frac3{2^3}\\ \frac3{2^4} \\ \frac1{2^4}\\ 0\\ 0\\ 0\\ 0\\ \vdots \end{array}\right),
\left( \begin{array}{c} \frac{10}{2^5} \\ \frac{10}{2^6} \\ \frac{4}{2^6}  \\ \frac{1}{2^6} \\ 0\\ 0\\ 0\\ \vdots \end{array}\right),
\left( \begin{array}{c} \frac{35}{2^7} \\ \frac{35}{2^8} \\ \frac{15}{2^8} \\ \frac{5}{2^8} \\ \frac{1}{2^8} \\ 0\\ 0\\ \vdots \end{array}\right),
\left( \begin{array}{c} \frac{126}{2^{9}} \\ \frac{126}{2^{10}}\\ \frac{56}{2^{10}}\\ \frac{21}{2^{10}}\\ \frac{6}{2^{10}}\\ \frac{1}{2^{10}}\\ 0\\ \vdots \end{array}\right).$$
Let us denote the numerator of the $i$-th entry of $\underline{v}^k$ by $n_{k,i}$.  We obtain the following relations:
\begin{equation*}
n_{k,i} =\begin{cases}  2n_{k-1, 1}+n_{k-1, 2}+n_{k-1, 3}+\dots +n_{k-1, n} & \text{for }  i=1, 2,\\
   n_{k-1, i-1}+n_{k-1, 2}+n_{k-1, 3}+\dots +n_{k-1, n} & \text{for }  3\le i\le k+1.
\end{cases}
\end{equation*}
Clearly the denominator is $2^{2k-1}$ for $i=1$ and $2^{2k}$ for $0\le i\le k+1$.
The claim follows by the observation that the formula for $p_{k, k-i+2}$ is the same as that for $n_{k, i}$.

Note that another way to prove this is by examining the recursive relations in Pascal's triangle - the terms $p_{k, k-i+2}$ can be observed on the $(k+1)$-st diagonal.
\end{proof}

Stirling's formula gives $p_{k,k} \approx (1/2) 2^{2k}/\sqrt{2\pi k}$.
 Therefore, the claim implies that $Z_k(\Phi_t,A_1) \ge C/\sqrt{k}$ for some $C>0$, and so the system is indeed recurrent.

To prove null recurrence, we appeal to Theorem~\ref{thm:RPF}.  Given $\rho_t$ the eigenfunction for $L_{\Phi_t}$ and $m_t$, the $(t, \Pconf(\Phi_t))$-conformal measure,  it suffices to show that $\int\rho_t~dm_t=\infty$.  These have been computed earlier and combine to give $\int_{A_i} \rho_t~dm_t = \left(\frac{\lambda^t}{1-\lambda^t}\right)^{i-1}$.  (Note we can rescale $m_t$ and $\rho_t$, but not in a way which would change our result.)  Since in this case $\lambda^t=1/2$, we obtain $\int \rho_t~dm_t=\infty$  as required.
\end{proof}

\end{document}